\documentclass[12pt,letterpaper]{article}

\pdfoutput=1

\usepackage{graphics,epsfig,graphicx,color}
\graphicspath{{../Figures/}}
\usepackage[caption = false]{subfig}
\usepackage{float}
\usepackage{capt-of}

\usepackage{amssymb,latexsym}

\usepackage{setspace}

\usepackage{amsmath}

\usepackage{mathrsfs,amsfonts}

\usepackage{amsthm,amsxtra}

\usepackage[final]{showkeys}

\usepackage{stmaryrd}

\usepackage{algorithm}
\usepackage{algorithmicx}
\usepackage{algpseudocode}

\usepackage{etaremune}
\usepackage{enumitem} 

\newif\ifPDF
\ifx\pdfoutput\undefined
	\PDFfalse
\else
	\ifnum\pdfoutput > 0
		\PDFtrue
	\else
		\PDFfalse
	\fi
\fi

\ifPDF
	\usepackage{pdftricks}
	\begin{psinputs}
		\usepackage{pstricks}
		\usepackage{pstcol}
		\usepackage{pst-plot}
		\usepackage{pst-tree}
		\usepackage{pst-eps}
		\usepackage{multido}
		\usepackage{pst-node}
		\usepackage{pst-eps}
	\end{psinputs}
\else
	\usepackage{pstricks}
\fi

\ifPDF
	\usepackage[debug,pdftex,colorlinks=true, 
	linkcolor=blue, bookmarksopen=false,
	plainpages=false,pdfpagelabels]{hyperref}
\else
	\usepackage[dvips]{hyperref}
\fi

\usepackage{fancyhdr}

\usepackage{comment}

\usepackage[toc,page]{appendix}

\usepackage{cancel}

\usepackage{multirow}

\newtheorem{theorem}{Theorem}[section]

\newtheorem{proposition}[theorem]{Proposition} 
 
\newtheorem{corollary}[theorem]{Corollary}

\newcommand{\dsum}{\displaystyle\sum}

\newcommand{\eps}{\varepsilon}

\newcommand{\fu}{\mathfrak{u}}

\newcommand{\fo}{\mathfrak{o}}

\newcommand{\fm}{\mathfrak{m}}

\newcommand{\fc}{\mathfrak{c}}

 \newcommand{\bbE}{\mathbb E}
 \newcommand{\bbN}{\mathbb N}
\newcommand{\bbP}{\mathbb P} 
\newcommand{\bbR}{\mathbb R}

\newcommand{\bnu}{{\boldsymbol \nu}}

 \newcommand{\bn}{\mathbf n}

 \newcommand{\bx}{\mathbf x}

 \newcommand{\bH}{\mathbf H}

\newcommand{\cC}{\mathcal C}  
\newcommand{\cE}{\mathcal E} \newcommand{\cF}{\mathcal F}
 \newcommand{\cH}{\mathcal H}

\newcommand{\wt}{\widetilde}
\newcommand{\wh}{\widehat}

\newenvironment{keywords}
{\noindent{\bf Key words.}\small}{\par\vspace{1ex}}
\newenvironment{AMS}
{\noindent{\bf AMS subject classifications 2010.}\small}{\par}

\newcommand{\bbeta}{\boldsymbol \beta}

\title{Characterizing impacts of model uncertainties in quantitative photoacoustics}

\author{
	Kui Ren\thanks{
		Department of Mathematics and the Institute for Computational Engineering and Sciences (ICES), The University of Texas, Austin, TX 78712;
		\href{mailto:ren@math.utexas.edu}{ren@math.utexas.edu}
		}
	\and
	Sarah Vall\'elian\thanks{
		Statistical and Applied Mathematical Sciences Institute,
		Research Triangle Park, NC 27709; \href{mailto:svallelian@samsi.info}{svallelian@samsi.info}
		}
}

\date{}

\begin{document}

\maketitle



\begin{abstract}
This work is concerned with uncertainty quantification problems for image reconstructions in quantitative photoacoustic imaging (PAT), a recent hybrid imaging modality that utilizes the photoacoustic effect to achieve high-resolution imaging of optical properties of tissue-like heterogeneous media. We quantify mathematically and computationally the impact of uncertainties in various model parameters of PAT on the accuracy of reconstructed optical properties. We derive, via sensitivity analysis, analytical bounds on error in image reconstructions in some simplified settings, and develop a computational procedure, based on the method of polynomial chaos expansion, for such error characterization in more general settings. Numerical simulations based on synthetic data are presented to illustrate the main ideas.
\end{abstract}






\section{Introduction}
\label{SEC:intro}

The field of uncertainty quantification has experienced tremendous growth in the past decade, with many efficient general-purpose computational algorithms developed and some specific theoretical issues mathematically understood; see, for instance, ~\cite{ArGhSo-JCP10,BiHo-IP08,BuGhMaSt-SIAM13,ChHeLu-SIAM06,DoEfHoLu-JCP06,EfHoLu-SIAM06,FlWiAkHiBlGh-SIAM11,GaFiWiGh-IJNME10,GaZa-JCP07a,HiKeCaCaRy-SIAM05,LeKn-Book10,LeKnNaGh-JCP04,LeNaPeGhKn-SIAM07,LiWiGh-SIAM10,NoPa-IP09,NoPaSt-LNCSE10,PeIaNo-Book15,RuCh-SAM05,Stuart-AN10,WiBi-SIAM08,ZaGa-JCP08} and references therein for some recent developments in the field. In this work, we investigate uncertainty quantification issues in image reconstruction problems in quantitative photoacoustic tomography (PAT), one of the recent hybrid imaging modality that combines the advantages of the classical ultrasound imaging and optical tomography~\cite{Beard-IF11,Wang-Book09,Wang-PIER14}. Our main focus is to characterize the impact of model uncertainties on the quality of the images reconstructed.

PAT is a coupled-physics imaging method that utilizes the photoacoustic effect to construct high-resolution images of optical properties of tissue-like heterogeneous media. In a typical experiment of PAT, we send a short pulse of near-infra-red (NIR) light into an optically heterogeneous medium, such as a piece of biological tissue. The photons travel inside the medium following a diffusion-type process. The medium absorbs a portion of the photons during the propagation process. The energy of the absorbed photons leads to temperature rise inside the medium which then results in thermal expansion of the medium. When the remaining photons exit, the medium cools down and contracts due to this temperature drop. The thermal expansion and contraction within the medium induces a pressure change which then propagates through the medium in the form of ultrasound waves. 

Let us denote by $X\subseteq\bbR^d$ ($d\ge 2$) the medium of interest and $\partial X$ its boundary, and denote by $u(\bx)$ the density of photons at position $\bx\in X$, integrated over the lifetime of the short light pulse sent into the medium. It is then well-known that $u(\bx)$ solves the following elliptic boundary value problem~\cite{Arridge-IP99,ArSc-IP09,BaRe-IP11,BaUh-IP10}:
\begin{equation}\label{EQ:Diff PAT}
	\begin{array}{rcll}
		-\nabla\cdot \gamma(\bx)\nabla u(\bx) + \sigma_a(\bx) u(\bx) &=& 0,& \mbox{in}\ \  X\\
		u(\bx) &=& g(\bx),& \mbox{on}\ \ \partial X
	\end{array}
\end{equation}
where $\gamma(\bx)>0$ and $\sigma_a(\bx)>0$ are the diffusion and absorption coefficients of the medium respectively, and $g$ is the model for the (time-integrated) illumination source. The initial pressure field generated by the photoacoustic effect is given as~\cite{BaUh-IP10}:
\begin{equation}\label{EQ:Data QPAT}
	H(\bx) = \Gamma(\bx) \sigma_a(\bx) u(\bx), \qquad \bx\in\bar X
\end{equation}
where $\Gamma$, usually called the Gr\"uneisen coefficient, is a function that describes the photoacoustic efficiency of the medium. The pressure field evolves, in the form of ultrasound, following the acoustic wave equation~\cite{BaUh-IP10,FiScSc-PRE07}:
\begin{equation}\label{EQ:Acous}
	\begin{array}{rcll}
		\dfrac{1}{c^2(\bx)}\dfrac{\partial^2 p}{\partial t^2} -\Delta p &=&0, 
			&\text{in}\ \bbR_+\times\bbR^d\\
		p(0,\bx)&=& H\chi_X,
			&\text{in}\ \bbR^d\\ 
		\dfrac{\partial p}{\partial t}(0,\bx)&=&0,&\text{in}\ \bbR^d
	\end{array}
\end{equation}
where $c$ is the speed of the ultrasound and $\chi_X$ is the characteristic function of the domain $X$. It turns out that change of optical properties has very small impact on the ultrasound speed $c(\bx)$. Therefore, $c(\bx)$ and the optical coefficients $\gamma(\bx)$ and $\sigma_a(\bx)$ are treated as independent functions.

In a PAT experiment, we measure the time-dependent ultrasound signal on the surface of the medium, 
\begin{equation}\label{EQ:PAT Data}
	y(t, \bx)= p_{|(0,T]\times\partial X}, 
\end{equation}
for a long enough time $T$. The objective is then to reconstruct one or more coefficients in the set $\{\Gamma(\bx), \sigma_{a}(\bx), \gamma(\bx)\}$ from these measurements. In general, data collected from multiple illumination sources are necessary when more than one coefficients are to be reconstructed.

Image reconstructions in PAT are often performed in two steps. In the first step, one reconstructs $H$ in the acoustic wave equation from measured ultrasound data~\cite{AgKuKu-PIS09,AgQu-JFA96,AmBrGaJu-SIAM12,AmBrJuWa-LNM12,BuMaHaPa-PRE07,CoArBe-IP07,FiHaRa-SIAM07,Haltmeier-SIAM11B,HaScSc-M2AS05,Hristova-IP09,KiSc-SIAM13,KuKu-EJAM08,Kunyansky-IP08,Nguyen-IPI09,PaSc-IP07,QiStUhZh-SIAM11,StUh-IP09,Tittelfitz-IP12}. Theory on uniqueness and stability of the inverse solutions, as well as analytical reconstruction strategies, have been developed in both the case of constant ultrasound speed and the case of variable ultrasound speed.

In the second step, one uses the functional $H$ as available internal data and attempts to reconstruct optical coefficients, mainly ($\Gamma, \sigma_a, \gamma$)~\cite{AmBoJuKa-SIAM10,BaJoJu-IP10,BaRe-CM11,BaUh-IP10,CoArKoBe-AO06,CoTaAr-CM11,GaOsZh-LNM12,LaCoZhBe-AO10,MaRe-CMS14,NaSc-SIAM14,PuCoArKaTa-IP14,ReGaZh-SIAM13,SaTaCoAr-IP13,ShCoZe-AO11,Zemp-AO10}. It has been shown that one can uniquely and stably reconstruct two of the three coefficients ($\Gamma, \sigma_a, \gamma$) if the third one is known~\cite{BaRe-IP11,BaUh-IP10,ReGaZh-SIAM13}. When multispectral data are available, one can simultaneously reconstruct all three coefficients uniquely and stably~\cite{BaRe-IP12} with additional assumption on the dependence of the coefficients on the wavelength.

All the aforementioned results in PAT rely on the assumption that the ultrasound speed $c(\bx)$ is \emph{known}. In practical applications, ultrasound speed inside the medium to be probe may not be known exactly. For instance in the imaging of biological tissues, it is often assumed that the ultrasound speed in tissues is the same as that in water. However, it is well-known now that ultrasound speed has about $15\%$ variations from tissue to tissue~\cite{YoKaHaYoSoCh-OE12}. Therefore, in PAT imaging of tissues, if we use the ultrasound speed of water in image reconstructions, the reconstructed images may not be the true images that we are interested in. They may contain artifacts caused by the inaccuracy of ultrasound speed used.

The objective of this work is exactly to characterize the impact of such inaccuracies in certain coefficients, which we will call \emph{uncertain coefficients} (for instance the ultrasound speed $c$) and denote by $\fu$, in the mathematical model on the reconstruction of other model coefficients, which we will call \emph{objective coefficients} (for instance the absorption coefficient $\sigma_a$) and denote by $\fo$. To explain the main idea, let us write abstractly the map from physical coefficients to the ultrasound data in PAT as
\begin{equation}\label{EQ:Model Gen}
	y=f(\fo,\fu),
\end{equation}
and denote by $\wt f^{-1}[\fu]$ an inversion algorithm that reconstruct $\fo$ with uncertainty coefficient $\fu$, then we are interested in estimating the relation between $\wt f^{-1}[\fu_1]\Big(f(\fo, \fu_1)\Big)-\wt f^{-1}[\fu_2]\Big(f(\fo, \fu_1)\Big)$ and $\fu_1 - \fu_2$. Whenever possible, we would like to derive stability results that bound errors in the reconstructions of $\fo$ with errors in the uncertainty coefficient $\fu$, that is, bounds of the type
\begin{equation}\label{EQ:Stab General}
	\|\wt f^{-1}[\fu_1]\Big(f(\fo, \fu_1)\Big)-\wt f^{-1}[\fu_2]\Big(f(\fo, \fu_1)\Big) \|_{X}\le \fc \|\wt\fu_1-\fu_2\|_{X'},\ \ \mbox{for some constant}\ \ \fc>0,
\end{equation}
with appropriately chosen function spaces $X$ and $X'$ (and the corresponding norms $\|\cdot\|_{X}$ and $\|\cdot\|_{X'}$). If such a bound can not hold, the problem is unstable under change of the uncertainty coefficient.

To take a closer look at the problem, let us assume that $f$ is sufficiently smooth in a neighborhood of some $(\fo_0, \fu_0)$. We can then simplify the problem by linearizing it in at $(\fo_0, \fu_0)$, when we know that the variation in $\fu$ is small. The linearization at background $\fu_0$ leads us to the system
\begin{equation}\label{EQ:Model Gen Lin}
	y=f(\fo_0, \fu_0)+\dfrac{\delta f}{\delta\fo}[\fo_0,\fu_0] \delta\fo + \dfrac{\delta f}{\delta\fu}[\fo_0,\fu_0] \delta\fu.
\end{equation}
This gives the following relation, after some straightforward algebra,
\begin{equation}\label{EQ:Model Gen Lin-2}
	\dfrac{\delta f}{\delta\fo}[\fo_0,\fu_0] \delta\fo = f\Big(\wt f^{-1}[\fu_0]\big(y-f(\fo_0, \fu_0)\big), \fu_0\Big)- \dfrac{\delta f}{\delta\fu}[\fo_0,\fu_0] \delta\fu.
\end{equation}
Therefore, in the linearized case, the uncertainty characterization that we intend to study boils down to the estimation of the size of the operator $(\dfrac{\delta f}{\delta\fo}[\fo_0, \fu_0])^{-1} \dfrac{\delta f}{\delta\fu}[\fo_0, \fu_0]$, assuming again that the linear operator $\dfrac{\delta f}{\delta\fo}[\fo_0, \fu_0]$ is invertible. Note that the first term on the right is the error in the datum caused by the inaccuracy of the reconstruction algorithm. It is not caused by uncertainty in $\fu$ and disappears when the reconstruction algorithm $\wt f^{-1}$ gives exactly the inverse of $f$ at $\fu_0$.

The rest of the paper is structured as follows. We first derive in Section~\ref{SEC:PAT} various qualitative bounds, in the form of~\eqref{EQ:Stab General}, on errors in PAT reconstructions of the objective coefficients due to errors in the uncertain coefficients. We then perform similar sensitivity analysis in Section~\ref{SEC:fPAT} for image reconstruction problems in fluorescence PAT, that is, photoacoustic tomography with fluorescent markers. To understand more quantitatively the uncertainty issues, we develop, in Section~\ref{SEC:Alg}, a computational algorithm that would allow us to build, numerically, the precise relation between $\|\wt f^{-1}[\fu_1]\Big(f(\fo, \fu_1)\Big)-\wt f^{-1}[\fu_2]\Big(f(\fo, \fu_1)\Big) \|_{X}$ and $\|\wt\fu_1-\fu_2\|_{X'}$. Numerical simulations based on synthetic ultrasound data are then presented, in Section~\ref{SEC:Simu}, to provide an overview of the impact of model uncertainties on the quality of image reconstructions in PAT and fPAT.

\section{Impact of model inaccuracies in PAT}
\label{SEC:PAT}

In this section, we study in detail some uncertainty characterization problems for PAT reconstructions of optical coefficients. Following the results in~\cite{BaRe-IP11}, we know that it is impossible to uniquely reconstruct all three coefficients $\Gamma$, $\sigma_a$ and $\gamma$ simultaneously. We will therefore focus only on the cases of reconstructing one or two coefficients. 

Throughout the rest of the paper, we denote by $L^p(X)$ ($1\le p\le \infty$) the usual space of Lebesgue integrable functions on $X$, $\cH^k(X)$ the Hilbert space of functions whose $j$th ($0\le j\le k$) derivatives are in $L^2(X)$. We denote by $\cC^k(X)$ whose derivatives up to $k$ are continuous in $X$. We will use $\|\cdot\|_{X}$ to denote the standard norm of function space $X$, and we denote by $\cF_\alpha$ the class of strictly positive functions bounded between two constants $\underline{\alpha}$ and $\overline{\alpha}$, 
\begin{equation}
\cF_{\alpha}=\{f(\bx):X\mapsto\bbR: 0< \underline{\alpha} \le f(\bx) \le \overline{\alpha} <\infty,\ \forall\bx\in X\}.
\end{equation}
We make the following general assumptions on the domain and the illumination source:\vskip 1mm
({\bf Ass-i}) the domain $X$ is bounded with smooth boundary $\partial X$, and ({\bf Ass-ii}) the boundary source $g$ is the restrictions of a $\cC^\infty$ function on $\partial X$, and $g(\bx)$ is selected such that the corresponding diffusion solution $u\ge \fc >0$ for some constant $\fc$.\vskip 2mm
\noindent It will be clear that the strong regularity assumptions on $X$ and $g$ can be relaxed significantly in the cases we consider. We made these assumptions simply to avoid the trouble of having to state conditions on them every time they are involved in a theoretical result. We emphasize that the assumption of having an illumination $g$ such that $u\ge \fc>0$ in $X$ is not unreasonable. In fact, with mild regularity and bound assumptions on the coefficients, the techniques developed in~\cite{AlDiFrVe-AdM17} allows us to show that when $g\ge \fc'>0$ for some constant $\fc'$ on $\partial X$, the solution to the diffusion equation satisfies $u\ge\fc>0$ for some $\fc$; see ~\cite{AlDiFrVe-AdM17,ReZh-SIAM18} for more discussions on this issue.


\subsection{Impact of inaccurate ultrasound speed}
\label{SUBSEC:C PAT}

We start with the impact of inaccurate ultrasound speed on optical reconstructions. This problem can be analyzed in a two-step fashion. In the first step, we analyze the impact of uncertainty in ultrasound speed on the reconstruction of the initial pressure field $H$. In the second step, we analyze the impact of the uncertainty in $H$ on the reconstruction of the optical coefficients. 

The first step of the analysis, i.e. the propagation of the uncertainty in ultrasound speed $c$ to the reconstructed initial pressure field $H$, has been studied by Oksanen and Uhlmann in~\cite{OkUh-MRL14}. Let $\Lambda_c$ be the operator defined through the relation
\begin{equation}
	p_{|(0, T]\times \partial X} = \Lambda_c H,
\end{equation}
where $p$ is the solution to the acoustic wave equation~\eqref{EQ:Acous}, then the following result is a simplified version of what is proved in~\cite{OkUh-MRL14}.
\begin{theorem}[Theorem 1 of~\cite{OkUh-MRL14}]\label{THM:C Stab H}
Let $\wt H$ and $H$ be the initial pressure field reconstructed from datum $\Lambda_c H$ under ultrasound speed $\wt c$ and $c$ respectively. Assume that $\|\wt H\|_{\cH^3(X)}, \|H\|_{\cH^3(X)}\le \fc_h$, $\|\wt c\|_{\cC^2(X)}, \|c\|_{\cC^2(X)}\le \fc_c$ for some constants $\fc_h$ and $\fc_c$. Then there exists $\eps_c$, $T$ and $\fc$ such that $\|\wt c-c\|_{\cC^1(X)}\le \eps_c$ implies  
\begin{equation}\label{EQ:C Stab H}
\|\wt H-H\|_{\cH^1(X)} \le \fc \|\wt c-c\|_{L^\infty(X)}\|\Lambda_c H\|_{\cH^1((0,T]\times\partial X)}^{1/2}.
\end{equation}
\end{theorem}
This conditional stability result basically says that, for relatively smooth ultrasound speed (at least $\cC^2$ to be more precise), when the uncertainty in the ultrasound speed $c$ is not too big, the error it induced in the reconstruction of the initial pressure field $H$ is also not big. This observation is, in some sense, confirmed by the numerical simulations in~\cite{DiReVa-IP15} where it is shown that one can make a reasonable error in the reconstruction of the ultrasound speed $c$ but still have a good reconstruction of the absorption coefficient $\sigma_a$ when simultaneous reconstruction of $c$ and $\sigma_a$ was performed.

\paragraph{The case of reconstructing $\Gamma$.} Let us first consider the (almost trivial) case of reconstructing the single coefficient $\Gamma$, assuming that all the other coefficients, besides the ultrasound speed $c$, are known exactly. The following result is straightforward to verify.
\begin{proposition}\label{PROP:C Stab Gamma}
	Let $\wt\Gamma\in\cC^3(X)\cap \cF_\alpha$ and $\Gamma\in\cC^3(X)\cap \cF_\alpha$ be the Gr\"uneisen coefficient reconstructed with ultrasound speeds $\wt c$ and $c$ respectively from ultrasound datum $\Lambda_c H$. Assume further that $\|\wt c\|_{\cC^2(X)}, \|c\|_{\cC^2(X)}\le \fc_c$ for some constants $\fc_h$ and $\fc_c$,  $\gamma \in\cC^2(X)\cap\cF_\alpha$ and $\sigma_a \in\cC^3(X)\cap\cF_\alpha$. Then there exists $\eps_c$, $T$ and $\fc$ such that $\|\wt c-c\|_{\cC^1(X)}\le \eps_c$ implies  
\begin{equation}\label{EQ:C Stab Upsilon}
\|\wt\Gamma-\Gamma\|_{\cH^1(X)}\le \fc \|\wt c-c\|_{L^\infty(X)}\|\Lambda_c H\|_{\cH^1((0,T]\times\partial X)}^{1/2}.
\end{equation}
\end{proposition}
\begin{proof}
With the assumptions in ({\bf Ass-i})-({\bf Ass-ii}) on the regularity and boundedness of $\sigma_a$, $\gamma$, $X$ as well as $g$, classical theory~\cite{Evans-Book10,GiTr-Book00} ensures that the diffusion equation~\eqref{EQ:Diff PAT} admits a unique bounded solution in $\cC^3(X)$ such that $0<\fc_1\le u(\bx)\le\fc_2$ for some constants $\fc_1$ and $\fc_2$. Therefore $\tilde\cH$ and $\cH$ satisfy the conditions in Theorem~\ref{THM:C Stab H}.

Moreover, we observe from the definition of $H$ in~\eqref{EQ:Data QPAT} that
\begin{equation}
	\wt H-H=(\wt\Gamma-\Gamma)\sigma_a(\bx) u(\bx).
\end{equation}
This relation then implies that
\begin{equation}\label{EQ:C Stab Upsilon 1}
	\|\wt\Gamma-\Gamma\|_{\cH^1(X)} \le \wt \fc \|\wt H-H\|_{\cH^1(X)}
\end{equation}
for some constant $\wt\fc$ that depends on the bounds of $\sigma_a$, $u$ as well as their gradients. The result in~\eqref{EQ:C Stab Upsilon} is then obtained by combining the bound ~\eqref{EQ:C Stab Upsilon 1} and the bound ~\eqref{EQ:C Stab H}.
\end{proof}
This simple exercise shows that the error, measured in $\cH^1$ norm, in the reconstruction of the Gr\"uneisen coefficient $\Gamma$, grows at most linearly, asymptotically, with respect to the maximal error we made in the ultrasound speed (which is again assumed to be relatively smooth). Therefore, if we use a relatively accurate ultrasound speed in our reconstructions of $\Gamma$, the errors in the reconstructions are relatively small.

\paragraph{The case of reconstructing $\sigma_a$.} We can reproduce the result for the reconstruction of the absorption coefficient, one of the most important quantity in practical applications. We have the following stability result.
\begin{theorem}\label{THM:C Stab Sigma}
Let $\wt \sigma_a\in\cC^3(X)\cap \cF_\alpha$ and $\sigma_a\in \cC^3(X)\cap \cF_\alpha$ be the absorption coefficients reconstructed with $\wt c$ and $c$ respectively from datum $\Lambda_c H$. In addition, assume that $\Gamma\in\cC^3(X)\cap\cF_\alpha$, $\gamma\in\cC^2(\bar X)\times\cF_\alpha$ and that $\|\wt c\|_{\cC^2(X)}, \|c\|_{\cC^2(X)}\le \fc_c$ for some constants $\fc_h$ and $\fc_c$. Then there exists $\eps_c$, $T$ and $\fc$ such that $\|\wt c-c\|_{\cC^1(X)}\le \eps_c$ implies
\begin{equation}\label{EQ:C Stab Sigma}
	\|\wt \sigma_a-\sigma_a\|_{\cH^1(X)}\le \fc \|\wt c-c\|_{L^\infty(X)}\|\Lambda_c H\|_{\cH^1((0,T]\times\partial X)}^{1/2}.
\end{equation}
\end{theorem}
\begin{proof}
Let $\wt u$ and $u$ be the solution to the diffusion equation~\eqref{EQ:Diff PAT} with coefficients $\wt \sigma_a$ and $\sigma_a$ respectively. We define $w=\wt u-u$. It is straightforward to verify that $w$ solves the following diffusion equation
\begin{equation}
\begin{array}{rcll}
	-\nabla\cdot \gamma \nabla w &=& -(\wt H-H)/\Gamma, & \mbox{in}\ \ X\\
	w &=& 0,& \mbox{on}\ \ \partial X
\end{array}
\end{equation}
With the boundedness assumptions on the coefficients $\gamma$ and $\Gamma$, we deduce directly from classical elliptic theory~\cite{Evans-Book10,GiTr-Book00} that
\begin{equation}\label{EQ:C Stab Sigma 1}
	\|w\|_{\cH^1(X)} \le \wt \fc_1 \|\wt H-H\|_{L^2(X)},
\end{equation}
for some constant $\wt\fc$. Meanwhile, we observe directly from the definition of datum $H$ that
\begin{equation}
(\wt H-H)/\Gamma=\wt\sigma w + (\wt\sigma_a-\sigma_a)u.
\end{equation}
This leads to the following bound, after using the fact that $u$ is positive and bounded away from zero,
\begin{equation}\label{EQ:C Stab Sigma 2}
\|\wt\sigma_a-\sigma_a\|_{\cH^1(X)}\le \fc_2 \big(\|\wt H-H\|_{\cH^1(X)}+\|w\|_{\cH^1(X)}\big)
\end{equation}
We can now combine~\eqref{EQ:C Stab Sigma 2}, ~\eqref{EQ:C Stab Sigma 1} and~\eqref{EQ:C Stab H} to obtain the bound in~\eqref{EQ:C Stab Sigma}.
\end{proof}

\paragraph{The case of reconstructing multiple coefficients.} The case of simultaneous reconstruction of more than one coefficients is significantly more complicated. The theory developed in~\cite{BaRe-IP11} states that one can reconstruct two of the three coefficients $(\Gamma, \sigma_a, \gamma)$ assuming that the third one is known. Multi-spectral data are need in order to simultaneous reconstruct all three coefficients~\cite{BaRe-IP12}. Let us define
\begin{equation}\label{EQ:Potential}
	\mu=\dfrac{\sqrt{\gamma}}{\Gamma\sigma_a} \qquad \qquad \mbox{and}\qquad \qquad 
q=-(\dfrac{\Delta\sqrt{\gamma}}{\sqrt{\gamma}}+\dfrac{\sigma_a}{\gamma}).
\end{equation}
We then have the following stability result.
\begin{theorem}\label{THM:C Stab Mu-Q}
Let $(\wt\Gamma, \wt\sigma_a, \wt\gamma)$ and $(\Gamma, \sigma_a, \gamma)$ be the coefficient pairs reconstructed with $\wt c$ and $c$ respectively, using data $\Lambda_c\bH=(\Lambda_c H_1, \Lambda_c H_2)$ generated from sources $g_1$ and $g_2$. Assume further that $\wt\gamma_{|\partial X}=\gamma_{|\partial X}$. Then, under the same conditions on $(\wt\Gamma, \wt\sigma_a, \wt\gamma, \wt c)$ and $(\Gamma, \sigma_a, \gamma, c)$ as in Theorem~\ref{THM:C Stab Sigma}, there exists $(g_1, g_2)$, $\eps_c$, $T$ and $\fc$ such that $\|\wt c-c\|_{\cC^1(X)}\le \eps_c$ implies 
\begin{multline}\label{EQ:C Stab Mu-Q}
	\|\wt q-q\|_{L^2(X)}+\|\wt\mu-\mu\|_{L^2(X)} \\ 
\le \fc \max\{\|\wt c-c\|_{L^\infty(X)}\|\Lambda_c \bH\|_{(\cH^1((0,T)\times\partial X))^2}^{1/2}, \|\wt c-c\|_{L^\infty(X)}^{\frac{4}{3d+12}}\|\Lambda_c \bH\|_{(\cH^1((0,T)\times\partial X))^2}^{\frac{2}{3d+12}}\}.
\end{multline}
\end{theorem}
\begin{proof}
Let $u_1$ and $u_2$ be the (positive) solutions to the diffusion equation~\eqref{EQ:Diff PAT} for sources $g_1$ and $g_2$ respectively. We multiply the equation for $u_1$ by $u_2$ and multiply the equation for $u_2$ by $u_1$. We take the difference of the results to get the following equation:
\begin{equation}
\begin{array}{rcll}
	-\nabla\cdot (\gamma u_1^2) \nabla\dfrac{u_2}{u_1} &=& 0,& \mbox{in}\ \ X\\
	\dfrac{u_2}{u_1} &=& \dfrac{g_2}{g_1},& \mbox{on}\ \ \partial X
\end{array}
\end{equation}
Using the fact that $H_1=\Gamma\sigma_a u_1$, and the fact that $u_2/u_1=H_2/H_1$, we can rewrite this equation as
\begin{equation}
\begin{array}{rcll}
	-\nabla\cdot \mu^2 \bbeta &=& 0,& \mbox{in}\ \ X\\
	\mu^2 &=& \mu_{|\partial X}^2,& \mbox{on}\ \ \partial X
\end{array}
\end{equation}
where $\bbeta=H_1^2\nabla\dfrac{H_2}{H_1}$ and $\mu_{|\partial X}^2=\gamma\dfrac{g_1^2}{H_{1|\partial X}^2}$. This is a transport equation for $\mu^2$ with \emph{known} vector field $\bbeta$. It is shown in~\cite{BaRe-IP11} that there exists a set of boundary conditions $(g_1, g_2)$ such that this transport equation admits a unique solution. Moreover, this transport equation for the unknown $\mu$ allows us to derive the following stability result for some constant $\fc_1$,
\begin{equation}\label{EQ:C Stab Mu-Q-1}
	\|\wt \mu-\mu\|_{L^\infty(X)}\le \fc_1 \|\wt \bH-\bH\|_{(L^2(X))^2}^{\frac{4}{3d+12}}.
\end{equation}

We now define $v_j=\sqrt{\gamma}u_j$ ($j=1,2$). It is well-known (and easy to verify) that $v_j$ solves the following elliptic partial differential equation:
\begin{equation}
\begin{array}{rcll}
	\Delta v_j(\bx)+ q(\bx) v_j(\bx) &=& 0,& \mbox{in}\ \ X\\
	v_j &=& \sqrt{\gamma_{|\partial X}}g_j ,& \mbox{on}\ \ \partial X
\end{array}
\end{equation}
Let $w_j=\wt v_j-v_j$ with $\wt v_j$ the solution to the above equation with $\wt q$, then $w_j$ solves
\begin{equation}
\begin{array}{rcll}
	\Delta w_j(\bx)+ \wt q(\bx) w_j(\bx) &=& -(\wt q-q) v_j,& \mbox{in}\ \ X\\
	w_j &=& 0,& \mbox{on}\ \ \partial X
\end{array}
\end{equation}
where the homogeneous boundary condition for $w_j$ comes from the assumption that $\wt\gamma_{|\partial X}=\gamma_{|\partial X}$. Since $0$ is not an eigenvalue of the operator $\Delta+\wt q$ (otherwise $0$ would be an eigenvalue of the operator $-\nabla\cdot \gamma \nabla + \wt\sigma_a$), and $u_j$ (therefore $v_j$) is positive and bounded away from zero, we conclude that~\cite{Evans-Book10,GiTr-Book00}:
\begin{equation}\label{EQ:C Stab Mu-Q-2}
	\fc_2 \|\wt q-q\|_{L^2(X)} \le \|w_j\|_{\cH^2(X)} \le \fc_3 \|\wt q-q\|_{L^2(X)}.
\end{equation}
for some constants $\fc_2$ and $\fc_3$.

To bound $w_j$ by the data, we observe that under the transform $v_j=\sqrt{\gamma}u_j$, we have $H_j=v_j/\mu$. Therefore,
\begin{equation}
	\mu\wt\mu \big(\wt H_j-H_j\big)=\mu w_j-(\wt \mu-\mu)v_j.
\end{equation}
This gives us the following bound for some constant $\fc_4$:
\begin{equation}\label{EQ:C Stab Mu-Q-3}
	\|w_j\|_{L^2(X)} \le \fc_4 \Big(\|\wt H_j-H_j\|_{L^2(X)}+\|\wt \mu-\mu\|_{L^2(X)}\Big).
\end{equation}
We can now combine~\eqref{EQ:C Stab Mu-Q-1},~\eqref{EQ:C Stab Mu-Q-2},~\eqref{EQ:C Stab Mu-Q-3} and~\eqref{EQ:C Stab H} to obtain the stability bound in~\eqref{EQ:C Stab Mu-Q}.
\end{proof}

\noindent{\bf Remark.} Note that the error bound we have in~\eqref{EQ:C Stab Mu-Q} is for the variables $\mu$ and $q$. This can be easily transformed into bounds on two of the triple $(\Gamma, \sigma_a, \gamma)$ if the third is known (since we can not reconstruct simultaneously all three optical coefficients according to~\cite{BaRe-IP11}). More precisely, we can replace the left hand side by $\|\wt\Gamma\wt\sigma_a-\Gamma\sigma_a\|_{L^2(X)}+\|\wt\sigma_a-\sigma_a\|_{L^2(X)}$ when $(\Gamma, \sigma_a)$ is to be reconstructed, by $\|\dfrac{\sqrt{\wt \gamma}}{\wt\sigma_a}-\dfrac{\sqrt{\gamma}}{\sigma_a}\|_{L^2(X)}+\|\dfrac{\Delta\sqrt{\wt \gamma}}{\sqrt{\wt \gamma}}-\dfrac{\Delta\sqrt{\gamma}}{\sqrt{\gamma}}\|_{L^2(X)}$ when $(\sigma_a, \gamma)$ is to be reconstructed and by $\|\wt \Gamma\sqrt{\gamma}-\Gamma\sqrt{\wt \gamma}\|_{L^2(X)}+\|\sqrt{\wt \gamma}\Delta\sqrt{\gamma}-\sqrt{\gamma}\Delta\sqrt{\wt \gamma}\|_{L^2(X)}$ when $(\Gamma, \gamma)$ is to be reconstructed.

\subsection{Impact of inaccurate diffusion coefficient}

We now study the impact of uncertainty in the diffusion coefficient $\gamma$ on the reconstruction of the other optical coefficients. Since both the uncertainty coefficient ($\gamma$) and the objective coefficients ($\Gamma$ and $\sigma_a$) are only involved in diffusion model~\eqref{EQ:Diff PAT}, we do not need to deal with the reconstruction problem in the first step of PAT. We therefore assume here that the internal datum $H$ is given.

\paragraph{The case of reconstructing $\Gamma$.} We again start with the reconstruction of the Gr\"uneisen coefficient $\Gamma$, assuming that $\sigma_a$ is known but $\gamma$ is not known. We have the following sensitivity result.
\begin{theorem}\label{THM:D Stab Gamma}
Let $\Gamma\in\cC^1(X)\cap \cF_\alpha$ and $\wt\Gamma \in\cC^1(X)\cap \cF_\alpha$ be the Gr\"uneisen coefficients reconstructed from datum $H$ with $\gamma\in\cC^1(\bar X)\cap \cF_\alpha$ and $\wt\gamma\in\cC^1(\bar X)\cap \cF_\alpha$ respectively. Then we have, for some constant $\fc$,
\begin{equation}\label{EQ:D Stab Gamma}
	\|\wt\Gamma-\Gamma\|_{\cH^1(X)} \le \fc \|\dfrac{H}{\Gamma\sigma_a}\|_{\cH^1(X)} \|\frac{\wt\gamma-\gamma}{\gamma}\|_{\cH^1(X)}.
\end{equation}
\end{theorem}
\begin{proof}
	Let $u$ and $\wt u$ be solutions to the diffusion equation~\eqref{EQ:Diff PAT} with coefficients $(\gamma, \sigma_a)$ and $(\wt\gamma, \sigma_a)$ respectively. Let us define $w=\wt u-u$. We then verify that $w$ solves
\begin{equation}
\begin{array}{rcll}
	-\nabla\cdot\wt\gamma \nabla w +\sigma_a w &=& \nabla\cdot (\wt\gamma-\gamma) \nabla u,& \mbox{in}\ \ X\\
	w &=& 0,& \mbox{on}\ \ \partial X
\end{array}
\end{equation}
This gives, following standard elliptic theory~\cite{Evans-Book10,GiTr-Book00}, the following bound:
\begin{equation}\label{EQ:D Stab Gamma-1}
	\|w\|_{\cH^1(X)}\le \fc_1 \|\nabla\cdot (\wt\gamma-\gamma)\nabla u\|_{L^2(X)}.
\end{equation}
Meanwhile, we observe, from the fact that the internal datum $H$ does not change with $\gamma$, that
\begin{equation}
	\wt \Gamma\sigma_a\wt u-\Gamma\sigma_a u = \wt\Gamma\sigma_a w+(\wt\Gamma-\Gamma)\sigma_a u =0.
\end{equation}
This, together with the fact that $u$ is positive and is bounded away from zero, gives us
\begin{equation}\label{EQ:D Stab Gamma-2}
	\| \wt\Gamma-\Gamma\|_{\cH^1(X)}\le \fc_2 \|w\|_{\cH^1(X)}.
\end{equation}
We can then combine~\eqref{EQ:D Stab Gamma-1} and~\eqref{EQ:D Stab Gamma-2} to get
\begin{equation}\label{EQ:D Stab Gamma-3}
	\| \wt\Gamma-\Gamma \|_{\cH^1(X)} \le \fc_3 \|\nabla\cdot (\wt\gamma-\gamma)\nabla u\|_{L^2(X)}.
\end{equation}

We now check the following calculations:
\begin{multline}
	\nabla\cdot(\wt\gamma-\gamma)\nabla u=\nabla\cdot\frac{\wt\gamma-\gamma}{\gamma} \gamma\nabla u = \frac{\wt\gamma-\gamma}{\gamma}\nabla\cdot\gamma\nabla u + \gamma\nabla u\cdot\nabla \frac{\wt\gamma-\gamma}{\gamma} \\ 
= \dfrac{H}{\Gamma} \frac{\wt\gamma-\gamma}{\gamma} + \gamma\nabla \frac{H}{\Gamma\sigma_a} \cdot\nabla \frac{\wt\gamma-\gamma}{\gamma} 
\end{multline}
where we have used the diffusion equation to replace $\nabla\cdot\gamma\nabla u$ by $\sigma_a u$ (which is simply $H/\Gamma$). This implies that
\begin{equation}
\int_{X}\big(\nabla\cdot (\wt\gamma-\gamma)\nabla u\big)^2 d\bx \le \fc_4 \big[ \int_X (\dfrac{H}{\Gamma\sigma_a} \frac{\wt\gamma-\gamma}{\gamma})^2 d\bx + \int_X |\nabla \frac{H}{\Gamma\sigma_a}|^2 |\nabla \frac{\wt\gamma-\gamma}{\gamma}|^2 d\bx.
\end{equation}
We can then apply the Cauchy-Schwarz inequality to conclude that
\begin{equation}\label{EQ:D Stab Gamma-4}
\|\nabla\cdot (\wt\gamma-\gamma)\nabla u\|_{L^2(X)}^2 \le \fc_5 \|\dfrac{H}{\Gamma\sigma_a}\|_{\cH^1(X)}^2 \|\frac{\wt\gamma-\gamma}{\gamma}\|_{\cH^1(X)}^2.
\end{equation}
The stability bound in~\eqref{EQ:D Stab Gamma} then follows from~\eqref{EQ:D Stab Gamma-3} and~\eqref{EQ:D Stab Gamma-4}.
\end{proof}

\paragraph{The case of reconstructing $\sigma_a$.} For the reconstruction of the absorption coefficient $\sigma_a$ assuming $\Gamma$ known, we can prove a similar sensitivity result.
\begin{theorem}
Let $\wt\sigma_a\in\cC^1(X)\cap\cF_\alpha$ and $\sigma_a\in\cC^1(X)\cap\cF_\alpha$ be the absorption coefficients reconstructed with $\wt\gamma\in\cC^1(\bar X)\cap\cF_\alpha$ and $\gamma\in\cC^1(\bar X)\cap\cF_\alpha$ respectively from datum $H$. Then, for some constant $\fc$, the following bound holds:
\begin{equation}\label{EQ:D Stab Sigma}
	\|\wt\sigma_a-\sigma_a \|_{\cH^1(X)} \le \fc \|\dfrac{H}{\Gamma\sigma_a}\|_{\cH^1(X)} \|\frac{\wt\gamma-\gamma}{\gamma}\|_{\cH^1(X)}.
\end{equation}
\end{theorem}
\begin{proof}
Let $\wt u$ and $u$ be solutions to the diffusion equation~\eqref{EQ:Diff PAT} with $(\wt \gamma, \wt\sigma_a)$ and $(\gamma, \sigma_a)$ respectively. Define $w=\wt u-u$. Then $w$ solves
\begin{equation}
\begin{array}{rcll}
	-\nabla\cdot\wt\gamma \nabla w &=& \nabla\cdot (\wt \gamma-\gamma) \nabla u,& \mbox{in}\ \ X\\
	w &=& 0,& \mbox{on}\ \ \partial X
\end{array}
\end{equation}
where we have used the fact that $H=\Gamma \wt \sigma_a\wt u=\Gamma \sigma_a u$. This again gives us the same bound as in~\eqref{EQ:D Stab Gamma-1}, that is,
\begin{equation}\label{EQ:D Stab Sigma-1}
	\|w\|_{\cH^1(X)}\le \fc_1 \|\nabla\cdot (\gamma-\gamma) \nabla u\|_{L^2(X)},
\end{equation}
by standard elliptic theory~\cite{Evans-Book10,GiTr-Book00}. Using~\eqref{EQ:D Stab Gamma-4}, we have
\begin{equation}\label{EQ:D Stab Sigma-2}
	\|w\|_{\cH^1(X)}\le \fc_2 \|\dfrac{H}{\Gamma\sigma_a}\|_{\cH^1(X)} \|\frac{\wt\gamma-\gamma}{\gamma}\|_{\cH^1(X)}.
\end{equation}

From the datum $H=\Gamma \wt \sigma_a\wt u=\Gamma \sigma_a u$, we check that
\begin{equation}
	\Gamma\wt\sigma_a\wt u-\Gamma\sigma_a u 
	= \Gamma\wt\sigma_a w+\Gamma(\wt\sigma_a-\sigma_a) u =0,
\end{equation}
which in turn gives us,
\begin{equation}\label{EQ:D Stab Sigma-3}
	\|\wt\sigma_a-\sigma_a\|_{\cH^1(X)} \le \fc_3 \|w\|_{\cH^1(X)}.
\end{equation}
The stability in~\eqref{EQ:D Stab Sigma} then follows from~\eqref{EQ:D Stab Sigma-2} and~\eqref{EQ:D Stab Sigma-3}.
\end{proof}
Let us emphasize here that the difference between the right hand side of~\eqref{EQ:D Stab Gamma} and that of~\eqref{EQ:D Stab Sigma} is that the $\sigma_a$ is known in~\eqref{EQ:D Stab Gamma} while $\Gamma$ is known in~\eqref{EQ:D Stab Sigma}.

\paragraph{The case of reconstructing $(\Gamma, \sigma_a)$.} In the case of simultaneous reconstruction of $\Gamma$ and $\sigma_a$, we can prove the following stability result following similar arguments as in Theorem~\ref{THM:C Stab Mu-Q}.
\begin{theorem}\label{THM:D Stab GS}
Let $(g_1, g_2)$ be a set of boundary illuminations such that the data $\bH=(H_1, H_2)$ generated from it uniquely determine $(\Gamma, \sigma_a)$ as in Theorem~\ref{THM:C Stab Mu-Q}. Let $(\wt\Gamma, \wt\sigma_a)$ and $(\Gamma, \sigma_a)$ be the coefficient pairs reconstructed with $\wt\gamma\in\cC^2(\bar X)\cap\cF_\alpha$ and $\gamma\in\cC^2(\bar X)\cap\cF_\alpha$ respectively from data set $\bH=(H_1, H_2)$. Then we have that, for some constants $\fc$ and $\wt \fc$,
\begin{multline}\label{EQ:D Stab GS}
\fc \|\sqrt{\wt \gamma}-\sqrt{\gamma}\|_{L^2(X)} \le \|\wt \sigma_a-\sigma_a\|_{L^2(X)}+\|\wt\Gamma\wt\sigma_a-\Gamma\sigma_a\|_{L^2(X)} \\ 
\le \wt\fc \Big(\|\wt\gamma-\gamma\|_{L^2(X)}+\|\frac{\Delta\sqrt{\wt \gamma}}{\sqrt{\wt \gamma}}-\frac{\Delta\sqrt{\gamma}}{\sqrt{\gamma}}\|_{L^2(X)}\Big).
\end{multline}
\end{theorem}
\begin{proof}
From the proof of Theorem~\ref{THM:C Stab Mu-Q}, we conclude that $\mu$ is reconstructed independent of the uncertain and objective coefficients. Therefore, we have
\begin{equation}\label{EQ:D Stab GS-1}
	\dfrac{\wt\Gamma\wt\sigma_{a}}{\sqrt{\wt \gamma}}-\dfrac{\Gamma \sigma_{a}}{\sqrt{\gamma}}=0 .
\end{equation} 
This gives immediately the bound,
\begin{equation}\label{EQ:D Stab GS-2}
	\fc_1 \|\sqrt{\wt \gamma}-\sqrt{\gamma}\|_{L^2(X)} \le \|\wt\Gamma\wt\sigma_{a}-\Gamma \sigma_{a}\|_{L^2(X)}\le \wt\fc_1 \|\sqrt{\wt\gamma}-\sqrt{\gamma}\|_{L^2(X)}.
\end{equation} 

Let $v_j=\sqrt{\gamma} u_j$ $(j=1,2)$ and $w_j=\wt v_1-v_1$. Then $w_j$ solves,
\begin{equation}\label{EQ:D Stab GS-3}
\begin{array}{rcll}
	\Delta w_j(\bx)+ \wt q(\bx) w_j(\bx) &=& -(\wt q-q) v_j,& \mbox{in}\ \ X\\
	w_j &=& 0,& \mbox{on}\ \ \partial X
\end{array}
\end{equation}
Meanwhile, $H_j=\Gamma\sigma_a u_j=\mu v_j =\wt\Gamma\wt\sigma_a \wt u_j= \mu \wt v_j$. This implies that $w_j=0$. Equation~\eqref{EQ:D Stab GS-3} then leads to $\wt q=q$, that is,
\begin{equation}\label{EQ:D Stab GS-4}
\dfrac{\Delta\sqrt{\wt \gamma}}{\sqrt{\wt \gamma}}+\dfrac{\wt \sigma_a}{\wt \gamma}=\dfrac{\Delta\sqrt{\gamma}}{\sqrt{\gamma}}+\dfrac{\sigma_a}{\gamma}.
\end{equation}
This translates directly to the following bound:
\begin{equation}\label{EQ:D Stab GS-5}
	\|\wt \sigma_a-\sigma_a\|_{L^2(X)} \le \fc_2 \Big(\|\wt\gamma-\gamma\|_{L^2(X)}+\|\frac{\Delta\sqrt{\wt \gamma}}{\sqrt{\wt \gamma}}-\frac{\Delta\sqrt{\gamma}}{\sqrt{\gamma}}\|_{L^2(X)}\Big).
\end{equation}
The stability estimate in~\eqref{EQ:D Stab GS} the follows from ~\eqref{EQ:D Stab GS-2} and~\eqref{EQ:D Stab GS-4}. 
\end{proof}

It is important to note that the proof of Theorem~\ref{THM:D Stab GS} is mainly based on the relations~\eqref{EQ:D Stab GS-1} and~\eqref{EQ:D Stab GS-4}. Therefore, we can use the same procedure to study impact of uncertainty in one of the coefficients on the reconstruction of the other coefficients. For instance, it is straightforward to derive the following results on the impact of the uncertainty of $\Gamma$ on reconstructing $(\gamma, \sigma_a)$, and the impact of the uncertainty in $\sigma_a$ on the reconstruction of $(\Gamma, \gamma)$.
\begin{corollary}
Under the same assumptions in Theorem~\ref{THM:D Stab GS}, let $(\wt\gamma, \wt\sigma_a)$ and $(\gamma, \sigma_a)$ be the coefficient pairs reconstructed with $\wt\Gamma\in\cC^2(\bar X)\cap\cF_\alpha$ and $\Gamma\in\cC^2(\bar X)\cap\cF_\alpha$ respectively. Then we have that, for some constants $\fc_1$ and $\wt \fc_1$,
\begin{multline}
	\fc_1 \|\wt \Gamma-\Gamma\|_{L^2(X)}\le \|\frac{\wt\sigma_a}{\sqrt{\wt\gamma}}-\frac{\sigma_a}{\sqrt{\gamma}}\|_{L^2(X)} + \|(\dfrac{\Delta\sqrt{\wt\gamma}}{\sqrt{\wt \gamma}}-\dfrac{\Delta\sqrt{\gamma}}{\sqrt{\gamma}})+\dfrac{\sigma_a}{\gamma}(\sqrt{\wt\gamma}-\sqrt{\gamma})\|_{L^2(X)} \\ \le \wt\fc_1 \|\wt \Gamma-\Gamma\|_{L^2(X)}.
\end{multline}
Let $(\wt\Gamma, \wt\gamma)$ and $(\Gamma, \gamma)$ be the coefficient pairs reconstructed with $\wt\sigma_a \in\cC^2(\bar X)\cap\cF_\alpha$ and $\sigma_a \in\cC^2(\bar X)\cap\cF_\alpha$ respectively. Then there exist constants $\fc_2$ and $\wt \fc_2$ such that
\begin{multline}
	\fc_2 \|\wt \sigma_a-\sigma_a\|_{L^2(X)} \le \|\dfrac{\wt\Gamma}{\sqrt{\wt\gamma}}-\dfrac{\Gamma}{\sqrt{\gamma}}\|_{L^2(X)} + \|(\dfrac{\Delta\sqrt{\wt\gamma}}{\sqrt{\wt \gamma}}-\dfrac{\Delta\sqrt{\gamma}}{\sqrt{\gamma}})+\sigma_a (\dfrac{1}{\wt\gamma}-\dfrac{1}{\gamma})\|_{L^2(X)} \\ \le \fc_2 \|\wt\sigma_a-\sigma_a\|_{L^2(X)}.
\end{multline}
\end{corollary}

\section{Impact of model inaccuracies in fluorescence PAT}
\label{SEC:fPAT}

We now extend the sensitivity analysis in the previous section to image reconstruction problems in quantitative photoacoustics for molecular imaging. In this setup, we are interested in imaging contrast agents inside the medium of interests. For instance, in fluorescence PAT (fPAT)~\cite{BuGrSo-NP09,RaDiViMaPeKoNt-NP09,RaNt-MP07,RaViNt-OL07,ReZh-SIAM13}, fluorescent biochemical markers are injected into the medium to be probed. The markers will then accumulate on certain targeted heterogeneities, for instance cancerous tissues, and emit near-infrared light (at wavelength $\lambda_m$) upon excitation by an external light source (at a different wavelength which we denote by $\lambda_x$). In the propagation process, both the excitation photons and the fluorescence photons can be absorbed by the medium. This absorption process then generates ultrasound signals following the photoacoustic effect we described previously.

The densities of the excitation photons and emission photons, denoted by $u_x(\bx)$ and $u_m(\bx)$ respectively, solve the following system of coupled diffusion equations~\cite{AmGaGi-QAM14,CoChDuRoScArScYo-OE07,ReZh-SIAM13,SoTaMcElNeFrAr-AO07}:
\begin{equation}\label{EQ:Diff fPAT}
	\begin{array}{rcll}
  	-\nabla\cdot \gamma_x(\bx)\nabla u_x(\bx) + (\sigma_{a,xi}+\sigma_{a,xf}) u_x(\bx) &=& 0,
  	&\mbox{in}\ \ X\\
  	-\nabla\cdot \gamma_m(\bx)\nabla u_m(\bx) + \sigma_{a,m}(\bx) u_m(\bx) &=& \eta \sigma_{a,xf} u_x(\bx),
  	&\mbox{in}\ \ X\\
    u_x(\bx) = g_x(\bx), \qquad   u_m(\bx) &=& 0, & \mbox{on}\ \ \partial X
	\end{array}	
\end{equation}
where the subscripts $x$ and $m$ are used to label the quantities at the excitation and emission wavelengths, respectively. The external excitation source is modeled by $g_x(\bx)$. The total absorption coefficient at the excitation wavelength consists of two parts, the intrinsic part $\sigma_{a,xi}$ that is due to the medium itself, and the fluorescence part $\sigma_{a,xf}$ that is due to the injected fluorophores of the biochemical markers. The fluorescence absorption coefficient $\sigma_{x,f}(\bx)$ is proportional to the concentration $\rho(\bx)$ and the extinction coefficient $\eps(\bx)$ of the fluorophores, i.e. $\sigma_{x,f}=\eps(\bx)\rho(\bx)$. The coefficient $\eta(\bx)$ is called the fluorescence quantum efficiency of the medium. The product of the quantum efficiency and the fluorophores absorption coefficient, $\eta \sigma_{x,f}$, is called the quantum yield.

The initial pressure field generated by the photoacoustic effect in this case is given as~\cite{ReZhZh-IP15,ReZh-SIAM13}:
\begin{equation}\label{EQ:Data QfPAT}
	H(\bx)=\Gamma(\bx) \Big((\sigma_{a,xi}+(1-\eta)\sigma_{a,xf}) u_x(\bx)+\sigma_{a,m} u_m(\bx)\Big).
\end{equation}
This consists of a part from the excitation wavelength and a part from the emission wavelength and the two parts can not be separated. Note that the component $\eta\sigma_{a,xf}u_x$ is subtracted from the excitation part in~\eqref{EQ:Data QfPAT} since this component is the part of the energy used to generate the emission light, as appeared in the second equation of~\eqref{EQ:Diff fPAT}.

The initial pressure field generated from the fluorescence photoacoustic effect evolves according to the same acoustic wave equation~\eqref{EQ:Acous}. The objective of fPAT is to determine the fluorescence absorption coefficient $\sigma_{a,xf}(\bx)$ (and therefore the spatial concentration of the fluorophores inside the medium, i.e. $\rho(\bx)$), and the quantum efficiency $\eta(\bx)$, whenever possible, from measured ultrasound signals on the surface of the medium. It is generally assumed that the coefficient pairs $(\gamma_x, \sigma_{a,xi})$ and $(\gamma_m, \sigma_{a,m})$ are known already, for instance from a PAT process at excitation wavelength and another PAT process at emission wavelength. We refer interested reader to~\cite{BuGrSo-NP09,RaDiViMaPeKoNt-NP09,RaNt-MP07,RaViNt-OL07,ReZh-SIAM13} for more detailed discussions on fPAT.

The objective of this section is to translate the uncertainty characterization we developed in the previous section to the case of fPAT. The main ideas of the derivation remains the same. However, the calculations are slightly more lengthy since we have to deal with system of diffusion equations as in~\eqref{EQ:Diff fPAT} instead of a single diffusion equation as in~\eqref{EQ:Diff PAT}. For more details on the mathematical modeling, as well as uniqueness results on image reconstructions, in fPAT, we refer to~\cite{ReZhZh-IP15,ReZh-SIAM13}. We make the following regularity assumptions on the background coefficients:
\[
	\Gamma\in\cC^3(X)\cap\cF_\alpha, \ \ \ (\gamma_x, \gamma_m)\in[\cC^2(\bar X)\cap \cF_\alpha]^2, \ \ \ (\sigma_{a,xi}, \sigma_{a,m})\in[\cC^3(X)\cap \cF_\alpha]^2 .
\]

\subsection{The ultrasound speed uncertainty}
\label{SUBSEC:C fPAT}

We start with the most important case, the stability of reconstructing the fluorescence absorption coefficient $\sigma_{a,xf}$ with respect to the ultrasound speed uncertainty. As in Section~\ref{SUBSEC:C PAT}, we will first derive stability of the reconstruction with respect to uncertainty in $H$ and then combine the result with the stability in~\eqref{EQ:C Stab H}. We have the following result.
\begin{theorem}\label{THM:C Stab fPAT Sigma}
 Let $\wt\sigma_{a,xf}\in\cC^3(X)\times\cF_\alpha$ and $\sigma_{a,xf}\in\cC^3(X)\times\cF_\alpha$ be the fluorescence coefficient reconstructed with ultrasound speeds $\wt c$ and $c$ respectively from datum $\Lambda_c H$. Assume that $\|\wt H\|_{\cH^3(X)}, \|H\|_{\cH^3(X)}\le \fc_h$, $\|\wt c\|_{\cC^2(X)}, \|c\|_{\cC^2(X)}\le \fc_c$ for some constants $\fc_h$ and $\fc_c$. Then there exists $\eps_c$, $T$ and $\fc$ such that $\|\wt c-c\|_{\cC^1(X)}\le \eps_c$ implies
\begin{equation}\label{EQ:C Stab fPAT Sigma}
\|(\wt\sigma_{a,xf}-\sigma_{a,xf})u_x\|_{L^2(X)}\le\fc \|\wt c-c\|_{L^\infty(X)}\|\Lambda_c H\|_{\cH^1((0,T)\times\partial X)}^{1/2}.
\end{equation}
\end{theorem}
\begin{proof}
Let $(\wt u_x, \wt u_m)$ and $(u_x,u_m)$ be the solution of the diffusion system~\eqref{EQ:Diff fPAT} with coefficients $\wt\sigma_{a,xf}$ and $\sigma_{a,xf}$ respectively. Define $(w_x, w_m)=(\wt u_x-u_x, \wt u_m-u_m)$. We then check that $(w_x, w_m)$ solves
\begin{equation}\label{EQ:C Stab fPAT Sigma-1}
	\begin{array}{rcll}
  	-\nabla\cdot \gamma_x\nabla w_x(\bx) + (\sigma_{a,xi}+\wt\sigma_{a,xf}) w_x(\bx) &=& -(\wt\sigma_{a,xf}-\sigma_{a,xf})u_x,
  	&\mbox{in}\ \ X\\
  	-\nabla\cdot \gamma_m\nabla w_m(\bx) + \sigma_{a,m}(\bx) w_m(\bx) &=& \eta\wt\sigma_{a,xf} w_x +\eta(\wt\sigma_{a,xf}-\sigma_{a,xf})u_x,
  	&\mbox{in}\ \ X\\
    w_x(\bx) = 0, \qquad  w_m &=& 0, & \mbox{on}\ \ \partial X
	\end{array}	
\end{equation}
From the datum~\eqref{EQ:Data QfPAT}, we deduce that
\begin{equation}\label{EQ:C Stab fPAT Sigma-2}
	\wt H-H=(\sigma_{a,xi}+(1-\eta)\wt\sigma_{a,xf})w_x+(1-\eta)(\wt\sigma_{a,xf}-\sigma_{a,xf})u_x+\sigma_{a,m} w_m
\end{equation}
This gives 
\begin{equation}\label{EQ:C Stab fPAT Sigma-3}
	\|(\wt\sigma_{a,xf}-\sigma_{a,xf})u_x\|_{L^2(X)}\le \fc_1(\|\wt H-H\|_{L^2(X)}+\|w_x\|_{L^2(X)}+\|w_m\|_{L^2(X)}).
\end{equation}

Using the relation~\eqref{EQ:C Stab fPAT Sigma-2}, we can now rewrite the system~\eqref{EQ:C Stab fPAT Sigma-1} as
\begin{equation}\label{EQ:C Stab fPAT Sigma-4}
	\begin{array}{rcll}
  	-\nabla\cdot \gamma_x\nabla w_x - \frac{\eta\sigma_{a,xi}}{1-\eta} w_x(\bx) &=& \frac{\sigma_{a,m}}{1-\eta}w_m-\frac{\wt H-H}{1-\eta},
  	&\mbox{in}\ \ X\\
  	-\nabla\cdot \gamma_m\nabla w_m + \frac{\sigma_{a,m}}{1-\eta} w_m(\bx) &=& -\frac{\eta\sigma_{a,xi}}{1-\eta} w_x +\frac{\eta}{1-\eta}(\wt H-H),
  	&\mbox{in}\ \ X\\
    w_x(\bx) = 0, \qquad  w_m &=& 0, & \mbox{on}\ \ \partial X
	\end{array}	
\end{equation}
This is a strongly elliptic system of equations. With the assumption on the regularity of the coefficients, we have the classical bound~\cite{McLean-Book00}:
\begin{equation}\label{EQ:C Stab fPAT Sigma-5}
\|w_x\|_{L^2(X)}+\|w_m\|_{L^2(X)}\le \fc_2 \|\wt H-H\|_{L^2(X)}.
\end{equation}
The stability bound~\eqref{EQ:C Stab fPAT Sigma} then follows from~\eqref{EQ:C Stab fPAT Sigma-3}, ~\eqref{EQ:C Stab fPAT Sigma-5} and~\eqref{EQ:C Stab H}.
\end{proof}

Let us emphasize that the weight function $u_x$, i.e. the density of the excitation photons, in the sensitivity relation~\eqref{EQ:C Stab fPAT Sigma} is very important and can not be removed. The appearance of $u_x$ in the sensitivity analysis is consistent with the following fact. If $u_x$ vanishes in a region inside the domain, the moleculars in the region would not be excited to emit new light. Therefore, the acoustic data we measured contain no information on the medium in the region. Thus, we can not hope to reconstruct any information inside the region, which is demonstrated here since in that case $(\tilde\sigma_{a,xf}-\sigma_{a,xf})u_x=0$ in the estimate.

\subsection{Uncertainty due to quantum efficiency}

In applications of fPAT, it is often assumed that the quantum efficiency of the medium is known. This is true for some well-understood medium, but not in general. In fact, in many cases of classical fluoresence optical tomography (FOT), researchers are interested in reconstructing the quantum efficiency as well. However, it is not possible to reconstruct both coefficients simultaneously because of the non-uniqueness in the FOT inverse problem. We now assume that the quantum efficiency $\eta$ is the uncertainty coefficient and attempt to the sensitivity of reconstructing $\sigma_{a,xf}$ with respect to changes in $\eta$. In this case, we assume that the ultrasound speed $c$ is known exactly so that we have access to an accurate $H$ directly.

\begin{theorem}\label{THM:Eta Stab fPAT Sigma}
Let $\wt\sigma_{a,xf}\in\cC^1(X)\cap\cF_\alpha$ and $\sigma_{a,xf}\in \cC^1(X)\cap\cF_\alpha$ be reconstructed from datum $H$ with coefficients $\wt\eta$ and $\eta$ respectively. Then the following holds for some constant $\fc$:
\begin{equation}\label{EQ:Eta Stab Sigmaf}
\|(\wt\sigma_{a,xf}-\sigma_{a,xf})u_x\|_{L^2(X)}\le \fc \|(\wt\eta-\eta)u_x\|_{L^2(X)}.
\end{equation}
\end{theorem}
\begin{proof}
Let $(\wt u_x, \wt u_m)$ and $(u_x,u_m)$ be the solution of the diffusion system~\eqref{EQ:Diff fPAT} with coefficients $(\wt\eta, \wt\sigma_{a,xf})$ and $(\eta, \sigma_{a,xf})$ respectively. Then $(w_x, w_m)=(\wt u_x-u_x, \wt u_m-u_m)$ solves
\begin{equation}\label{EQ:Eta Diff QfPAT}
	\begin{array}{rcll}
  	-\nabla\cdot \gamma_x\nabla w_x + (\sigma_{a,xi}+\wt\sigma_{a,xf}) w_x(\bx) &=& -(\wt\sigma_{a,xf}-\sigma_{a,xf})u_x,
  	&\mbox{in}\ \ X\\
  	-\nabla\cdot \gamma_m\nabla w_m + \sigma_{a,m}(\bx) w_m(\bx) &=& \wt\eta\wt\sigma_{a,xf} w_x +(\wt\eta\wt\sigma_{a,xf}-\eta\sigma_{a,xf})u_x,
  	&\mbox{in}\ \ X\\
    w_x(\bx) = 0, \qquad  w_m(\bx) &=& 0, & \mbox{on}\ \ \partial X
	\end{array}	
\end{equation}
From the datum~\eqref{EQ:Data QfPAT}, we deduce that
\begin{equation}\label{EQ:Eta Data Rel2}
	(\sigma_{a,xi}+(1-\wt\eta)\wt\sigma_{a,xf})w_x+(\eta-\wt\eta)\wt\sigma_{a,xf} u_x+(1-\eta)(\wt\sigma_{a,xf}-\sigma_{a,xf})u_x+\sigma_{a,m} w_m =0
\end{equation}
This gives 
\begin{equation}\label{EQ:Eta Stab Sigmaf 1}
	\|(\wt\sigma_{a,xf}-\sigma_{a,xf})u_x\|_{L^2(X)}\le \fc_1(\|(\wt \eta-\eta)u_x\|_{L^2(X)}+\|w_x\|_{L^2(X)}+\|w_m\|_{L^2(X)}).
\end{equation}

Using the relation~\eqref{EQ:Eta Data Rel2}, we can now rewrite~\eqref{EQ:Eta Diff QfPAT} as
\begin{equation}\label{EQ:Eta Diff QfPAT 2}
	\begin{array}{rcll}
  	-\nabla\cdot \gamma_x\nabla w_x-(\frac{\eta\sigma_{a,xi}+(\eta-\wt\eta)\wt\sigma_{a,xf}}{1-\eta}) w_x &=& \frac{\sigma_{a,m}}{1-\eta} w_m+\frac{\eta-\wt\eta}{1-\eta}\wt\sigma_{a,xf} u_x,
  	&\mbox{in}\ \ X\\
  	-\nabla\cdot \gamma_m\nabla w_m + \frac{1-\eta+\wt\eta}{1-\eta}\sigma_{a,m} w_m &=& \frac{(\wt\eta-\eta)\wt\eta\wt\sigma_{a,xf}-\wt\eta\sigma_{a,xi}}{1-\eta} w_x -\wt\eta\frac{(\eta-\wt\eta)\wt\sigma_{a,xf}}{1-\eta} u_x,
  	&\mbox{in}\ \ X\\
    w_x(\bx) = 0, \qquad  w_m(\bx) &=& 0, & \mbox{on}\ \ \partial X
	\end{array}	
\end{equation}
This is again a strongly elliptic system of equation. With the bound and regularity assumptions on the coefficients, we deduce that~\cite{McLean-Book00}:
\begin{equation}\label{EQ:Eta Stab Sigmaf 2}
\|w_x\|_{L^2(X)}+\|w_m\|_{L^2(X)}\le \fc_2 \|(\wt\eta-\eta)u_x\|_{L^2(X)}.
\end{equation}
By combining the stability in~\eqref{EQ:Eta Stab Sigmaf 1}, ~\eqref{EQ:Eta Stab Sigmaf 2}, we arrive at the stability bound in~\eqref{EQ:Eta Stab Sigmaf}.
\end{proof}

\subsection{The impact of partial linearization}

One of the main difficulties in imaging fluorescence is how to eliminate the strong background light. One way in practice is to take the background out by simulating the background distribution with the diffusion model for the propagation of excitation light inside the medium. However, due to the presence of $\sigma_{a,xf}$ in the first diffusion equation in~\eqref{EQ:Diff fPAT}, one can not simply solve that equation for its solution since $\sigma_{a,xf}$ is unknown. In many applications, it is simply assumed that $\sigma_{a,xf}$ is small so that it can be dropped from the equation for the excitation light. This is roughly speaking a partial linearization of the original model. 

We now characterize the impact of this partial linearization on the reconstruction of the fluorescence absorption coefficient.
\begin{theorem}
Let $\sigma_{a,xf}\in\cC^1(X)\cap\cF_\alpha$ and $\wt\sigma_{a,xf}\in\cC^1(X)\cap\cF_\alpha$ be the absorption coefficients reconstructed from the diffusion model~\eqref{EQ:Diff fPAT} and its partially-linearization (i.e. by setting $\sigma_{a,xf}=0$ in the x-component of the diffusion system) from a given data set. Then there exists a constant $\fc$ such that
\begin{equation}\label{EQ:Model Stab Sigmaf}
\|(\wt\sigma_{a,xf}-\sigma_{a,xf})u_x\|_{L^2(X)}\le \fc \|\sigma_{a,xf}u_x\|_{L^2(X)}.
\end{equation}
\end{theorem}
\begin{proof}
Let $(u_x,u_m)$ be the solution to the diffusion system~\eqref{EQ:Diff fPAT} and $(\wt u_x, \wt u_m)$ be the solution to the partially linearized system (with coefficient $\wt\sigma_{a,xf}$). Then $(w_x, w_m)=(\wt u_x-u_x, \wt u_m-u_m)$ solves the following system:
\begin{equation}\label{EQ:Model Diff QfPAT}
	\begin{array}{rcll}
  	-\nabla\cdot \gamma_x(\bx)\nabla w_x(\bx) + \sigma_{a,xi}(\bx) w_x(\bx) &=& \sigma_{a,xf} u_x,
  	&\mbox{in}\ \ X\\
  	-\nabla\cdot \gamma_m(\bx)\nabla w_m(\bx) + \sigma_{a,m}(\bx) w_m(\bx) &=& \eta\wt\sigma_{a,xf} w_x +\eta(\wt\sigma_{a,xf}-\sigma_{a,xf})u_x,
  	&\mbox{in}\ \ X\\
    w_x(\bx) = 0, \qquad  w_m &=& 0, & \mbox{on}\ \ \partial X
	\end{array}	
\end{equation}
From the datum~\eqref{EQ:Data QfPAT}, we find the relation:
\begin{equation}\label{EQ:Model Data Rel}
	(\sigma_{a,xi}+(1-\eta)\wt\sigma_{a,xf})w_x+(1-\eta)(\wt\sigma_{a,xf}-\sigma_{a,xf})u_x+\sigma_{a,m} w_m =0,
\end{equation}
which leads immediately to the following bound for some constant $\fc_1$:
\begin{equation}\label{EQ:Model Stab Sigmaf 1}
	\|(\wt\sigma_{a,xf}-\sigma_{a,xf})u_x)\|_{L^2(X)}\le \fc_1(\|w_x\|_{L^2(X)}+\|w_m\|_{L^2(X)}).
\end{equation}

Meanwhile, using the relation~\eqref{EQ:Model Data Rel}, we can rewrite the system~\eqref{EQ:Model Diff QfPAT} as
\begin{equation}\label{EQ:Model Diff QfPAT 2}
	\begin{array}{rcll}
  	-\nabla\cdot \gamma_x(\bx)\nabla w_x(\bx) + \sigma_{a,xi} w_x(\bx) &=& \sigma_{a,xf}u_x,
  	&\mbox{in}\ \ X\\
  	-\nabla\cdot \gamma_m(\bx)\nabla w_m(\bx) + \frac{\sigma_{a,m}}{1-\eta} w_m(\bx) &=& -\frac{\eta\sigma_{a,xi}}{1-\eta} w_x,
  	&\mbox{in}\ \ X\\
    w_x(\bx) = 0, \qquad  w_m &=& 0, & \mbox{on}\ \ \partial X
	\end{array}	
\end{equation}
The bound and regularity assumptions on its coefficient assure that the solutions to this strongly elliptic system has the following stability bound:
\begin{equation}\label{EQ:Model Stab Sigmaf 2}
\|w_x\|_{L^2(X)}+\|w_m\|_{L^2(X)} \le \fc_2 \|\sigma_{a,xf}u_x\|_{L^2(X)},
\end{equation}
with $\fc_2$ a constant. The stability bound~\eqref{EQ:Model Stab Sigmaf} then follows from~\eqref{EQ:Model Stab Sigmaf 1} and ~\eqref{EQ:Model Stab Sigmaf 2}.
\end{proof}

\section{Numerical uncertainty quantification}
\label{SEC:Alg}

We now implement a computational procedure, based on the computational uncertainty quantification machinery developed in the past years, for a more quantitative characterization of impact of uncertainties in quantitative photoacoustics. Our main focus here is on developing new computational techniques for general uncertainty quantification problems, but rather on the application of existing methods to PAT and fPAT image reconstruction problems. In a nutshell, we model $\fu$ as a random process, following some given probability law. We then construct a large population of random samples of $\fu$, and evaluate the corresponding inverse solutions $\fo$. Once we have these random samples of $\fo$, we study its statistics, mainly average and variance since we do not have efficient ways to visualize the sample distribution.

We assumed here that we can collect ultrasound data from $N_s\ge 2$ optical illuminations sources $\{g^s\}_{s=1}^{N_s}$ for the inverse problems (to ensure that we have enough data for unique reconstructions of the objective coefficients).

\subsection{Generalized polynomial chaos approximation}
\label{SUBSEC:Dim}

Let $(\Omega, \cF, \bbP)$ be an abstract probability space. We model our uncertainty coefficient by a random process $\fu(\bx, \omega)$, $(\bx, \omega)\in X\times\Omega$, that satisfies $0<\underline{\fu}\le \fu(\bx, \omega)\le \overline{\fu}<+\infty$ for some $\underline{\fu}$ and $\overline{\fu}$. To make the uncertainty quantification problem computationally feasible, that is, to reduce the dimension of the space of admissible uncertainty coefficients, we restrict ourselves to the class of random processes that admit a simple spectral representation. 

To be more precise, let $\xi(\omega): \Omega \mapsto \bbR$ be a uniform random variable with density function $\mu(\xi)$ and $\{\phi_{k}\}$ the family of probability Legendre polynomials, orthogonal with respect to the weight $\mu(\xi)$. We assume that the uncertainty coefficient $\fu$ is well-approximated by the following $K_\fu+1$ term truncated generalized polynomial chaos~\cite{Najm-ARFM09,XiKa-SIAM02}:
\begin{equation}\label{EQ:PCE fu}
	\fu(\bx, \xi(\omega)) = \sum_{k=0}^{K_\fu} \wh\fu_k(\bx) \phi_k(\xi(\omega)).
\end{equation}
For the purpose of simplifying the presentation, we assume that the polynomial bases are normalized in the sense that $\bbE\{\phi_k(\xi) \phi_{k'}(\xi) \} = \delta_{kk'}$. Interested readers are referred to~\cite{LeKn-Book10,MaNa-JCP09,Najm-ARFM09,PeIaNo-Book15,XiKa-SIAM02} for detailed discussions on representing random variables of different types using appropriate orthogonal polynomials. 

With the representation~\eqref{EQ:PCE fu}, we can generate random samples of $\fu$ once we know the coefficient functions $\{\wh\fu_k \}_{k=0}^{K_\fu}$ which do not depend on realizations. 

Let us emphasize that the sample uncertainty coefficients we constructed from~\eqref{EQ:PCE fu} has to satisfy the regularity and bounds requirements we imposed on the uncertainty coefficients. The regularity requirements in the space variable are satisfied by imposing smoothness on the coefficient functions $\{\wh\fu_k\}$. To satisfy the bounds requirements, we perform a linear rescaling on $\fu$. More precisely, assuming that $\fu$ generated by ~\eqref{EQ:PCE fu} satisfies $\underline{\fm}\le \fu(\bx,\omega)\le \overline{\fm}$, we perform $\frac{\overline{\fu}-\underline{\fu}}{\overline{\fm}-\underline{\fm}}\fu(\bx, \omega)+\frac{\overline{\fu}\underline{\fm}-\underline{\fu}\overline{\fm}}{\overline{\fm}-\underline{\fm}}\to \fu(\bx, \omega)$ to put $\fu$ in the range $[\underline{\fu}, \overline{\fu}]$.

\subsection{Constructing model predictions}

Once we know how to construct samples of the uncertainty coefficient, we need to solve inverse problems with these samples to compute the corresponding objective coefficients. We do this in two steps, described in this section and the next one respectively.

For each sample of the uncertainty coefficient $\fu(\bx, \xi)$, we need to evaluate the corresponding acoustic data predicted by the mathematical models with this uncertainty coefficient and the true objective coefficient which we denote by $\fo_t$: $y=f(\fu(\bx, \xi), \fo_t)$. The most accurate way of doing this is to solve the diffusion equation~\eqref{EQ:Diff PAT} (or the diffusion system~\eqref{EQ:Diff fPAT} in fPAT)) and then the acoustic wave equation~\eqref{EQ:Acous} for each realization of $\fu(\bx, \xi)$ (and the true objective coefficient $\fo_t$). However, this approach is computationally too expensive when a large number of samples need to be constructed.

Here we take advantage of the fact that, under the regularity assumptions of the coefficients involved, the solutions to the mathematical models in PAT and fPAT, therefore also the acoustic data predicted, are sufficiently regular with respect to these coefficients; see for instance~\cite[Lemma 2.1]{DiReVa-IP15}. Therefore, when these coefficients are smooth with respect to the random variable $\xi$, the solution to the equations are also sufficiently smooth with respect to the random variable. 


The smooth dependence of the solutions to the diffusion equation and the acoustic wave equation on the random variable $\xi$ indicates that these solutions can be represented efficiently using polynomial chaos representations. Let $K_u$ a positive integer, and 
\begin{equation}\label{EQ:PCE Diff}
u^s(\bx, \xi) = \sum_{k=0}^{K_u} \wh u_k^s(\bx) \, \phi_k(\xi)
\end{equation}
be the truncated polynomial chaos expansion of the diffusion solution with source $g^s$ ($1\le s\le N_s$). Using the standard projection procedure, we verify that $\wh u_k^s$ solves the following coupled diffusion system, $1\le k\le K_u$, $1\le s\le N_s$:
\begin{equation}\label{EQ:Diff PAT PCE}
	\begin{array}{rcll}
		-\dsum_{k'=1}^{K_u} \nabla\cdot \gamma_{kk'} \nabla \wh u_{k'}^s(\bx) + \dsum_{k'=1}^{K_u} \sigma_{kk'} \wh u_{k'}^s(\bx) &=& 0,& \mbox{in}\ \  X\\
		\wh u_k^s(\bx) &=& \wh g_k^s(\bx),& \mbox{on}\ \ \partial X
	\end{array}
\end{equation}
where 
\[
\gamma_{kk'}(\bx) = \sum_{j=0}^{K_\gamma} w_{kk'j} \wh \gamma_j(\bx), \ \ \  \sigma_{kk'}(\bx) = \sum_{j=0}^{K_\sigma} w_{kk'j} \wh \sigma_j(\bx),\ \ \ \mbox{ and }\ \ \wh g_k^s(\bx)=w_k\ g^s(\bx) 
\]
with the weights defined as $w_{kk'j}=\bbE\{\phi_k \phi_{k'} \phi_j\}$ and $w_k=\bbE\{\phi_k\}$. The functions $\{\wh \gamma_j\}_{j=0}^{K_\gamma}$  and $\{\wh \sigma_j\}_{j=0}^{K_\sigma}$ are the coefficients in the truncated polynomial chaos representation of $\gamma$ and $\sigma$ in the form of~\eqref{EQ:PCE fu}. 

This system of diffusion equations allows us to solve for $\wh u_k^s(\bx)$ as functions of the polynomial chaos expansion of the coefficients $\gamma$ and $\sigma_a$, which then allow us to to construct random samples of $u^s(\bx, \omega)$ following the polynomial chaos expansion~\eqref{EQ:PCE Diff}.

In the same manner, let $K_p$ be a positive integer and
\begin{equation}\label{EQ:PCE Acous}
	p^s(t, \bx, \omega) = \sum_{k=0}^{K_p} \wh p_k^s(t,\bx) \, \phi_k(\xi)
\end{equation}
be the truncated polynomial chaos expansion of the ultrasound pressure field. We then verify that the functions $\wh p_k(t, \bx)$ solves the following coupled system of acoustic wave equations:
\begin{equation}\label{EQ:Acous PCE}
	\begin{array}{rcll}
		\dsum_{k'=0}^{K_p} c_{kk'}\dfrac{\partial^2 \wh p_k^s}{\partial t^2} -\Delta \wh p_k^s &=&0, 
			&\text{in}\ \bbR_+\times\bbR^d\\
		\wh p_k^s(0,\bx)&=& \wh H_k^s \chi_X,
			&\text{in}\ \bbR^d\\ 
		\dfrac{\partial \wh p_k^s}{\partial t}(0,\bx)&=&0,&\text{in}\ \bbR^d
	\end{array}
\end{equation}
where 
\[
	c_{kk'}=\sum_{j=0}^{K_c} w_{kk'j} \wh c_j \ \ \ \mbox{and}\ \ \ \wh H_k^s(\bx)=\sum_{k'=0}^{K_\Gamma} \sum_{j=0}^{K_\sigma} \sum_{i=0}^{K_u} w_{kk'ji} \wh \Gamma_{k'} \wh\sigma_j \wh u_i^s
\]
with $\{\wh c_j\}_{j=0}^{K_c}$ and $\{\wh \Gamma_j\}_{j=0}^{K_\Gamma}$ being the coefficients in the polynomial chaos expansion of $\frac{1}{c^2(\bx, \xi)}$ and $\Gamma$ respectively, and the weights $w_{kk'ji}=\bbE\{\phi_k\phi_{k'}\phi_j \phi_i\}$.

The system of equations~\eqref{EQ:Diff PAT PCE} and ~\eqref{EQ:Acous PCE} now enable us to compute the PCE coefficients of ultrasound data from given PCE coefficients for the uncertainty coefficients involved, i.e., a subset of $\{\wh c_k\}_{k=1}^{K_c}$, $\{\wh \Gamma_k\}_{k=1}^{K_\Gamma}$, $\{\wh \gamma_k\}_{k=1}^{K_\gamma}$ and $\{\wh \sigma_k\}_{k=1}^{K_\sigma}$.

\subsection{Evaluating uncertainty in objective coefficients}

The next step is to study how the uncertainty in the data caused by the inaccuracy in the uncertainty coefficient is propagated into the objective coefficient that we are interested in reconstructing. 

The most accurate way of doing this is to solve the inverse problem for each realization of the uncertainty coefficient and study the distribution of the reconstructed coefficients. In terms of the abstract formulation in ~\eqref{EQ:Model Gen}, this means that we solve
\[
	f(\fo, \fu_t)=y(\omega) \equiv f(\fo_t, \fu(\bx, \omega))
\]
for $\fo$ for each $\omega$. This is computationally intractable for practical purpose. Bayesian type of inversion methods, such as these developed in~\cite{KaSo-Book05,MaZa-IP09,Stuart-AN10,WaZa-IP05}, are alternative ways to study such uncertainty quantification problems. The main issue here is that to apply these Bayesian methods, we need to be able to evaluate the likelihood function for each given candidate objective coefficient $\fo$. This is again computationally very hard to do since we do not have an explicit formula for the likelihood function which is the law of the ``noise'', $f(\fo_t, \fu(\bx, \omega))-f(\fo_t, \fu_t)$. We only have samples of the noise, as we constructed in the previous section. Fitting these samples into a known parameterized distribution with an explicit expression, for the instance the multi-dimensional Gaussian distribution,  is possible but would require that the exact form of the distribution been known {\em a priori}, which is hard to do here due to the high nonlinearity of the map $\fu \mapsto \fo(\fu)$.

Here we propose a method that is again based on the polynomial chaos representation: we represent the objective coefficient with polynomial chaos and reconstruct the coefficient of the representation directly from the data represented by the polynomial chaos coefficients $\{\{\wh p_k^{s*}(t,\bx)\}_{k=0}^{K_p}\}_{s=1}^{N_s}$. 

\paragraph{Step I.} The first step is to propagate the uncertainty from the acoustic data, $\{\{\wh p_k^*(t,\bx)\}_{k=0}^{K_p}\}_{s=1}^{N_s}$, into the initial pressure field under the true ultrasound speed $c_0$. We perform this using a time-reversal strategy~\cite{Hristova-IP09}. Let $t'=T-t$, we solve the coupled wave equations, $0\le k\le K_p$, $1\le s\le N_s$:
\begin{equation}\label{EQ:Acous PCE TR}
	\begin{array}{rcll}
		\dfrac{1}{c_0^2(\bx)}\dfrac{\partial^2 \wh q_k^s}{\partial {t'}^2} -\Delta \wh q_k^s &=&0, 
			&\text{in}\ (0, T] \times X \\
		\wh q_k^s(0,\bx)&=& 0,
			&\text{in}\ X \\ 
		\dfrac{\partial \wh q_k^s}{\partial t'}(0,\bx)&=&0,&\text{in}\ X \\
		\wh q_k^s(t',\bx)&=& \wh p_k^{s*}(t',\bx),&\text{on}\ \partial X\\
	\end{array}
\end{equation}
with true ultrasound speed $c_0$ until time $t'=T$ to reconstruct the coefficients of the polynomial chaos expansion of the initial pressure field $H$:
\[
	\wh H_k^{s*}(\bx) = \wh q_k^s(T, \bx), \ \ 0\le k\le K_p,\ 1\le s\le N_s.
\]
In our numerical simulations, we take measurement $T$ long enough to ensure a faithful reconstruction of the PCE coefficients of the initial pressure fields $\{\{\wh H_k^{s*}(\bx)\}_{k=1}^{K_p}\}_{s=1}^{N_s}$.

\paragraph{Step II.} The next step is to propagate the uncertainty in reconstructed initial pressure field $H$ to the objective coefficients to be reconstructed. We solve this problem via a least-square procedure. For instance, in the case where we are interested in reconstructing $(\Gamma, \sigma_a)$, treating $\gamma$ as the uncertainty coefficient, we reconstruct the coefficients $\{\wh \Gamma_k\}_{k=0}^{K_\Gamma}$ and  $\{\wh \sigma_k\}_{k=0}^{K_\sigma}$ as the solution to the following minimization problem:
\begin{equation}\label{EQ:LS Min}
	\min_{\{\wh \Gamma_j\}_{j=0}^{K_\Gamma}, \{\wh \sigma_k\}_{k=0}^{K_\sigma} }\dfrac{1}{2} \sum_{s=1}^{N_s}\sum_{i=0}^{K_p} \left\| \sum_{j=0}^{K_\Gamma} \sum_{k=0}^{K_\sigma} \sum_{k'=0}^{K_u} w_{ijkk'} \wh \Gamma_j \wh\sigma_{k} \wh u_{k'}^s -\wh H_i^{s*}\right\|_{L^2(X)}^2
\end{equation}
subject to the constraints, $0\le i\le K_u$, $1\le s\le N_s$:
\begin{equation}\label{EQ:Diff PAT PCE Recon}
	\begin{array}{rcll}
		-\nabla\cdot \gamma_0(\bx)\nabla \wh u_i^s(\bx) + \dsum_{k'=0}^{K_u}  \dsum_{k=0}^{K_\sigma} w_{ik'k} \wh\sigma_{k} \wh u_{k'}^s &=& 0,& \mbox{in}\ \  X\\
		\wh u_i^s(\bx) &=& \wh g_i^s(\bx),& \mbox{on}\ \ \partial X
	\end{array}
\end{equation}
where $\gamma_0$ is the true diffusion coefficient and the weights $w_{ikk'}$ and $w_{ijkk'}$ are defined the same way as before.

We solve the least-square minimization problem~\eqref{EQ:LS Min} with a quasi-Newton method based on the Broyden-Fletcher-Goldfarb-Shanno (BFGS) rule for Hessian update that we implemented in~\cite{ReBaHi-SIAM06}. We will not describe in detail this classical optimization algorithm but refer interested readers to~\cite{NoWr-Book06} for in-depth discussions on theoretical and practical aspects of the algorithm.

In Algorithm~\ref{ALG:UC}, we outline our implementation of the uncertainty quantification procedure for the case where $\fu=\gamma$ is the uncertainty coefficient and $\fo=(\Gamma, \sigma_a)$ is the objective coefficient. We need to change the algorithm only slightly for other combinations of uncertainty and objective coefficients. 
\begin{algorithm}
	\caption{Numerical Uncertainty Characterization Procedure} 
	\begin{algorithmic}[1]
		\State Set the PCE coefficients for $\gamma(\bx,\omega)$, $\{\wh \gamma_k\}_{k=0}^{K_\gamma}$
		\For{$s=1$ to $N_s$} 
		\State Solve the forward diffusion model~\eqref{EQ:Diff PAT PCE} with illumination source $g^s$
		\State Construct PCE coefficients for the initial pressure field, i.e. $\{H_k^s(\bx)\}_{k=0}^{K_p}$
		\State Solve the coupled system~\eqref{EQ:Acous PCE} with $\{H_k^s(\bx)\}_{k=0}^{K_p}$ for $\{\wh p_k^{s*}(t, \bx)\}_{k=0}^{K_p}$
		\State Reverse time for data $\{\wh p_k^{s*}(t, \bx)\}_{k=0}^{K_p}$
		\State Solve the coupled wave equations ~\eqref{EQ:Acous PCE TR} to reconstruct $\{H_k^{s*}\}_{k=0}^{K_p}$
		\EndFor		
		\State Solve the minimization problem~\eqref{EQ:LS Min} to reconstruct $\left(\{\wh \Gamma_j\}_{j=0}^{K_\Gamma}, \{\wh \sigma_k\}_{k=0}^{K_\sigma}\right)$
		\State Perform statistics on $\Gamma$ and $\sigma_a$ using their reconstructed PCE coefficients
	\end{algorithmic}
	\label{ALG:UC}
\end{algorithm}

\section{Numerical simulations}
\label{SEC:Simu}

We now present some numerical simulations,  following the computational procedure that we presented in Section~\ref{SEC:Alg}, to illustrate the main ideas of this work. We focus on two-dimensional simulations and select the simulation domain to be the square $X=(0, 2) \times (0, 2)$. 

To avoid solving the acoustic wave equation~\eqref{EQ:Acous} in unbounded domain $\bbR^2$, we replace~\eqref{EQ:Acous} with the same equation in $X$ with Neumann boundary condition. The measured data is now the solution of the wave equation on $\partial X$. We discretize the wave equation with a standard second order finite difference scheme in both spatial and temporal variables.

For the diffusion equation~\eqref{EQ:Diff PAT}, we use a first-order finite element method on an unstructured triangular mesh. The quantities on the triangular mesh are interpolated onto the uniform mesh, and vice versa, using a high-order interpolation scheme when needed. In all the simulations we performed, we verified, through mesh refining, that the interpolation errors are much smaller than the discretization error.

To construct samples of uncertainty coefficients, we observe from the polynomial chaos representation~\eqref{EQ:PCE fu} that the mean and variance of $\fu$ are given respectively as
\[
	\bbE\{\fu\} = \wh\fu_0 \qquad \mbox{and}\qquad {\mathbb Var}\{\fu\} = \sum_{k=1}^{K_\fu}  \wh \fu_k ^2.
\]
This gives us simple ways to control the mean and the variance of the random uncertainty coefficients. In our simulations, we take $\wh \fu_0$ as the true value of the uncertainty coefficient and add randomness as perturbations to $\wh\fu_0$, through the coefficient functions $\{\wh \fu_k\}_{k=1}^{K_\fu}$. We consider two types of random perturbations: the ones that are smooth in space and the ones that are piecewise smooth (in a special way) in space. 

\paragraph{Spatially smooth uncertainty coefficients.} To construct spatially smooth perturbations to the uncertainty coefficients, we take the PCE coefficients $\{\wh \fu_k\}_{k=1}^{K_c}$ as linear combinations of the Laplace-Neumann eigenfunctions on domain $X$. To be precise, let $(\lambda_\bn, \varphi_\bn)$ ($\bn=(n, m)\in \bbN_0\times \bbN_0$) be the eigenpair of the eigenvalue problem:
\begin{equation*}\label{EQ:Laplace}
	-\Delta \varphi = \lambda \varphi, \quad \mbox{in}\ \ X, \qquad \bnu\cdot\nabla \varphi = 0, \quad \mbox{on}\ \ \partial X.
\end{equation*}
Then $\lambda_\bn=\left(\dfrac{n\pi}{2}\right)^2+\left(\dfrac{m\pi}{2}\right)^2$, and
\[
	\varphi_\bn (x, y) = \cos(\frac{n\pi}{2} x) \cos(\frac{m\pi}{2} y) .
\]
In our numerical simulations, we take 
\begin{equation}\label{EQ:Smooth Uncert Coeff}
	\wh \fu_k = \sum_{n+m=k} c_\bn\ \varphi_\bn(x,y), \ \ 1\le k\le K_\fu
\end{equation}
with $\{c_\bn\}$ uniform random variables in $[-1, 1]$. Note that $\{c_\bn\}$ are fixed once they are generated. They do not change during the later stage of the uncertainty quantification process. Once the coefficients $\{\wh \fu_k\}_{k=1}^{K_\fu}$ are generated, we perform a linear scaling on them to get the variance of $\fu$ to the size that we need.

\paragraph{Piecewise smooth uncertainty coefficients.} To construct piecewise smooth perturbations to the uncertainty coefficients, we take the PCE coefficients $\{\wh \fu_k\}_{k=1}^{K_c}$ as linear combinations of the characteristic functions of $J$ randomly-placed disks in $X$. That is,
\begin{equation}\label{EQ:Disc Uncert Coeff}
	\wh \fu_k  = \sum_{j=1}^J c_{k, j} \chi_{D_j}(\bx), \qquad D_j=\{ \bx \ |\ |\bx-\bx_j|\le r_j\}, \ \ 1\le k\le K_\fu .
\end{equation}
As in the previous case, the centers $\{\bx_k \}_{k=1}^{K_\fu}$ and the radii $\{r_k\in[0.1,\ 0.2]\}_{k=1}^{K_\fu}$ of the disks, as well as the weights $\{c_{k,j}\}$ (uniform random variables in $[-1, 1]$) for the linear combinations, are fixed once they are generated. They do not change during the later stage of the uncertainty quantification process. We also rescale the amplitude of the perturbations to control the size of the variance of the perturbations. Note that, the theoretical analysis in the previous sections needs the uncertainty coefficients to be sufficiently smooth. In our numerical simulations, however, we try to neglect this smoothness requirement to see what would happen if the uncertainty coefficients are discontinuous, as long as the equations involved are still numerically solvable.

\subsection{Ultrasound speed uncertainty}

We first present some simulations on the reconstruction of optical coefficients under uncertain ultrasound speeds.
\begin{figure}[ht]
	\centering
	\includegraphics[angle=0,width=0.24\textwidth]{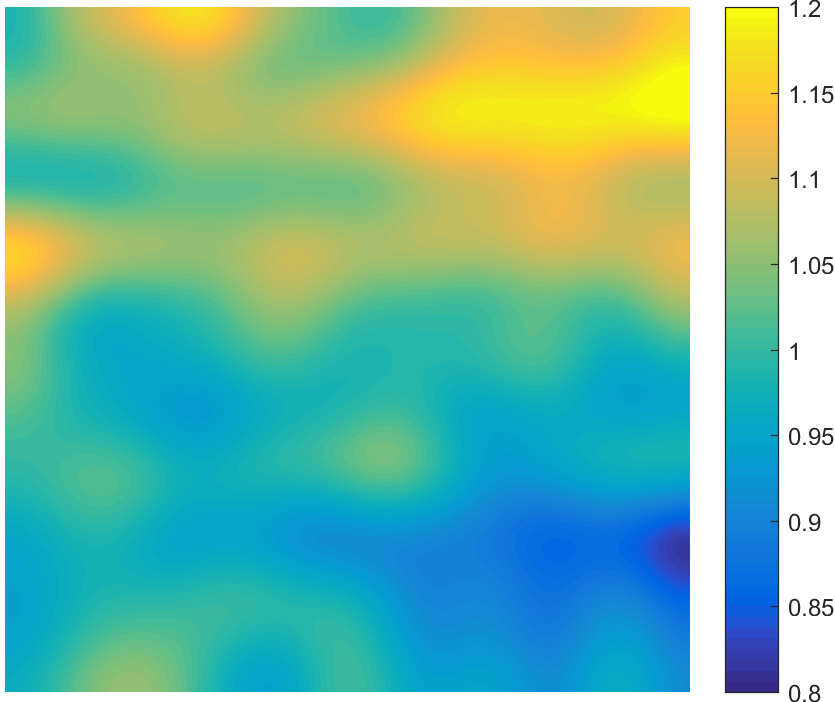} 
	\includegraphics[angle=0,width=0.24\textwidth]{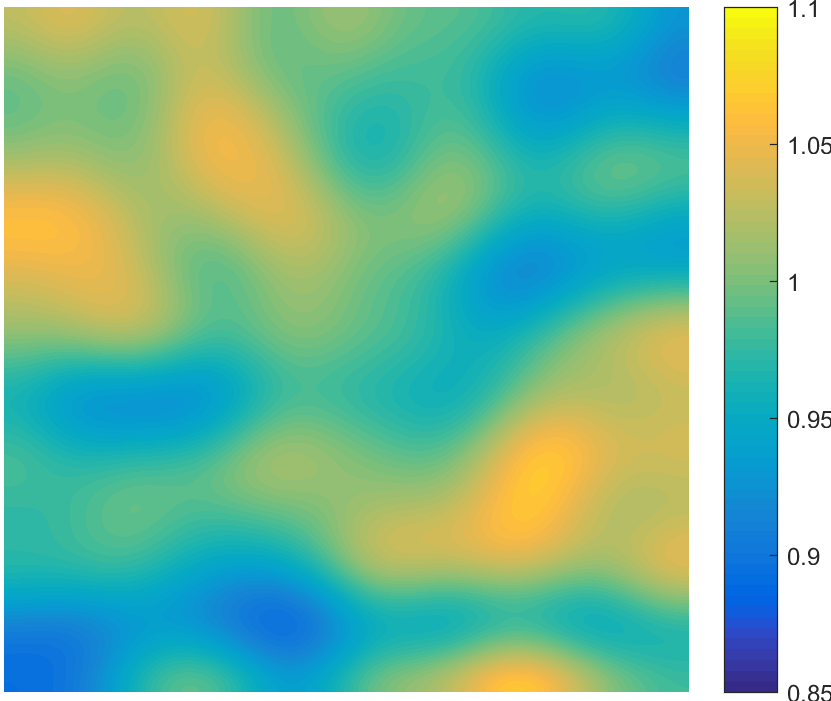}
	\includegraphics[angle=0,width=0.24\textwidth]{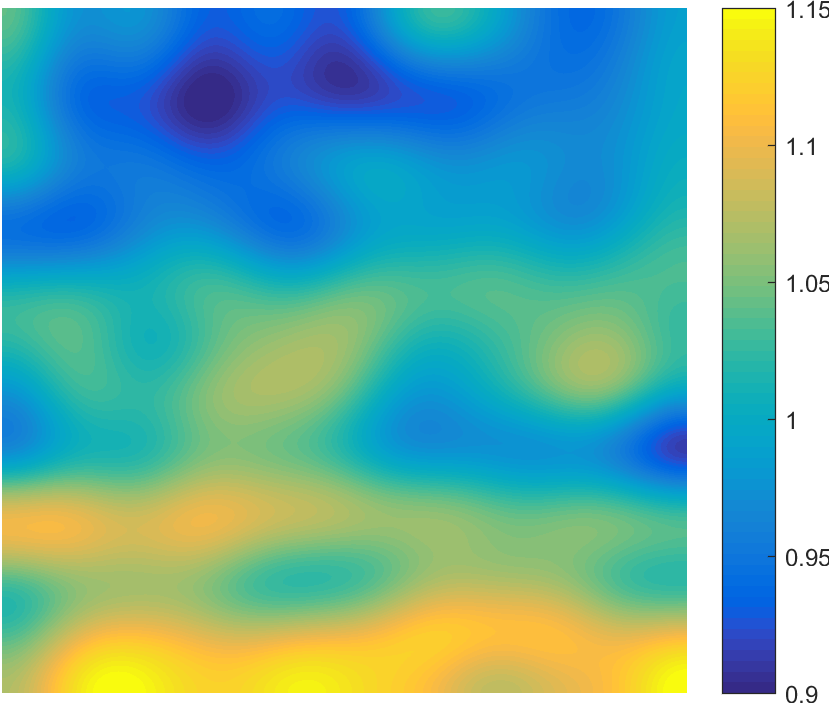} 
	\includegraphics[angle=0,width=0.24\textwidth]{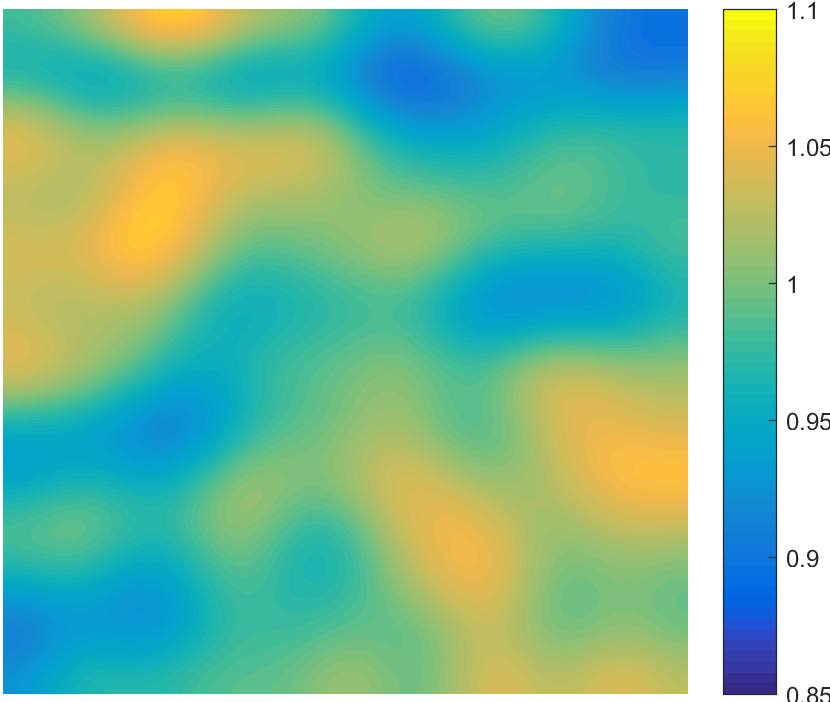} \\
	\includegraphics[angle=0,width=0.24\textwidth]{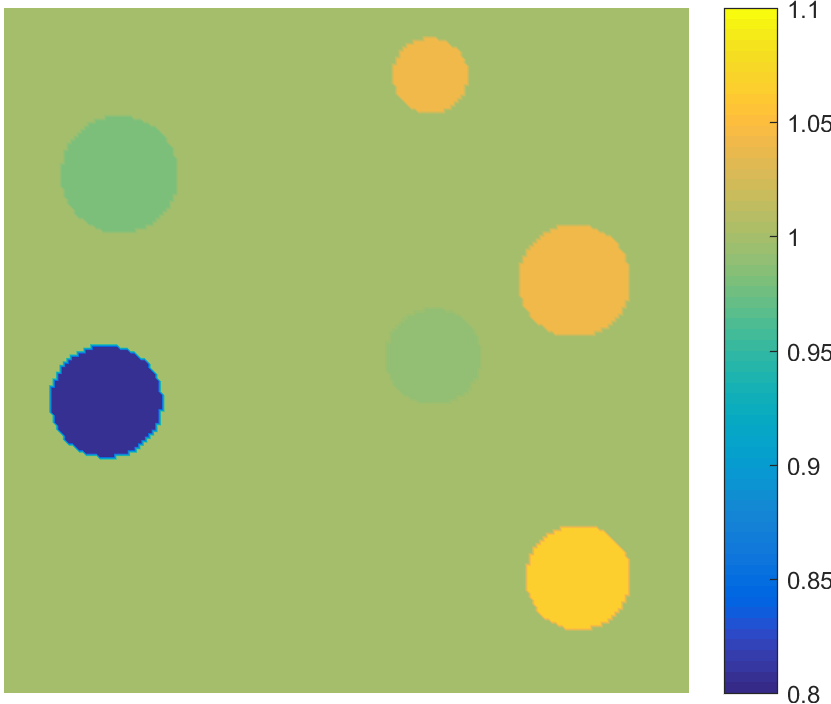}
	\includegraphics[angle=0,width=0.24\textwidth]{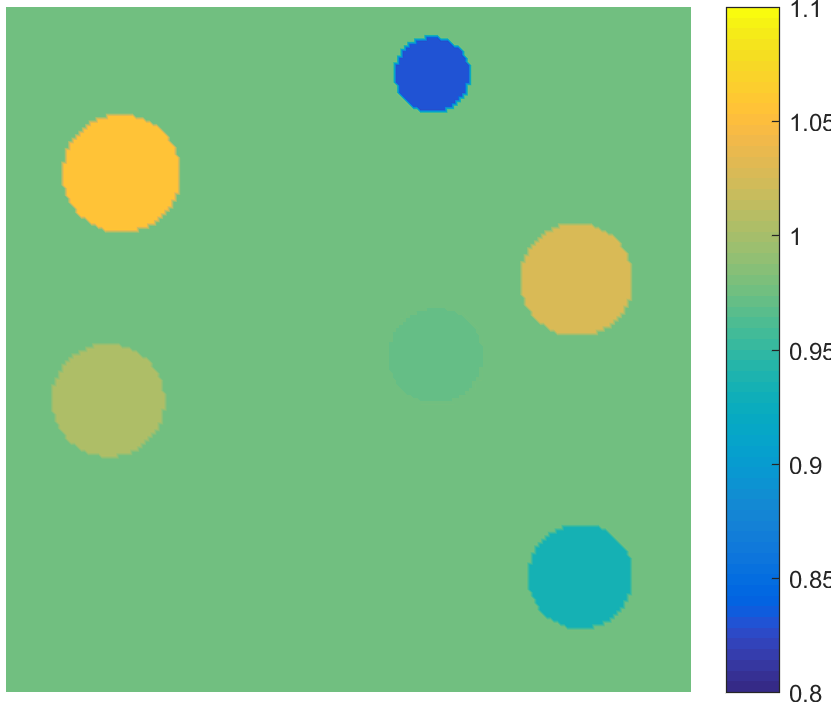}
	\includegraphics[angle=0,width=0.24\textwidth]{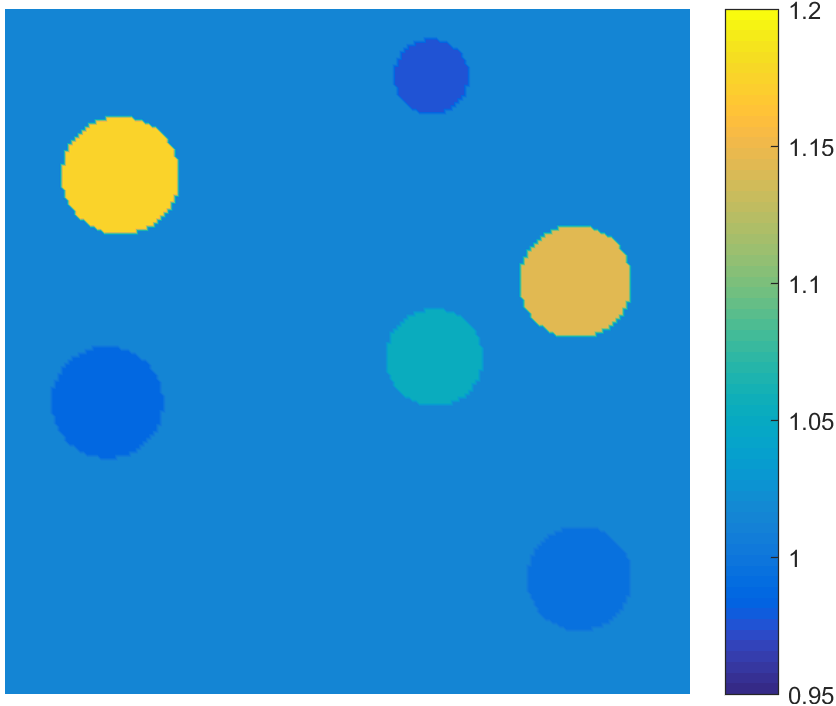}
	\includegraphics[angle=0,width=0.24\textwidth]{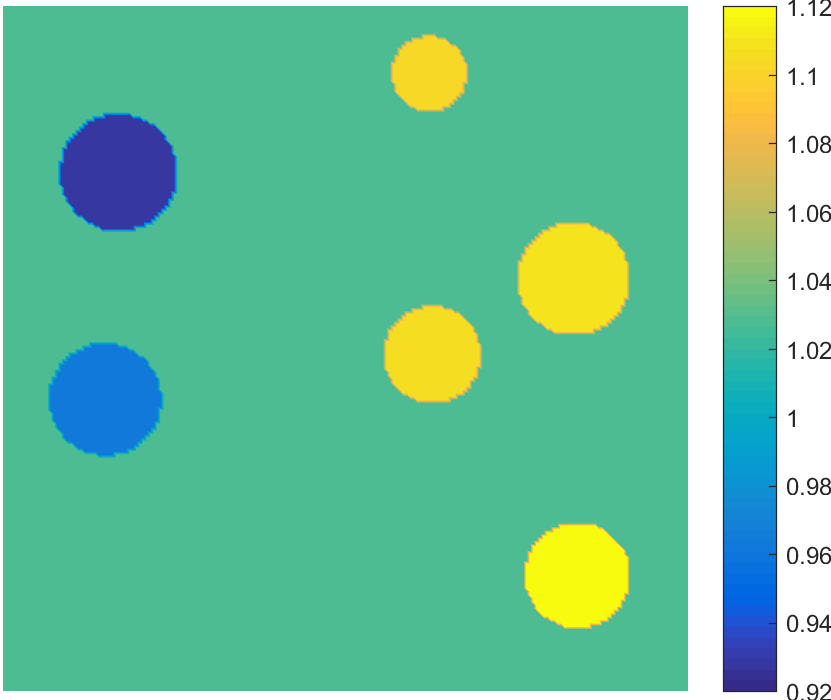}
	\caption{Typical realizations of (i) smooth (top row) and (ii) piecewise smooth (bottom row) ultrasound speed function.}
	\label{FIG:Ultrasound Speed}
\end{figure}

\begin{figure}[ht]
	\centering
	\includegraphics[angle=0,width=0.25\textwidth]{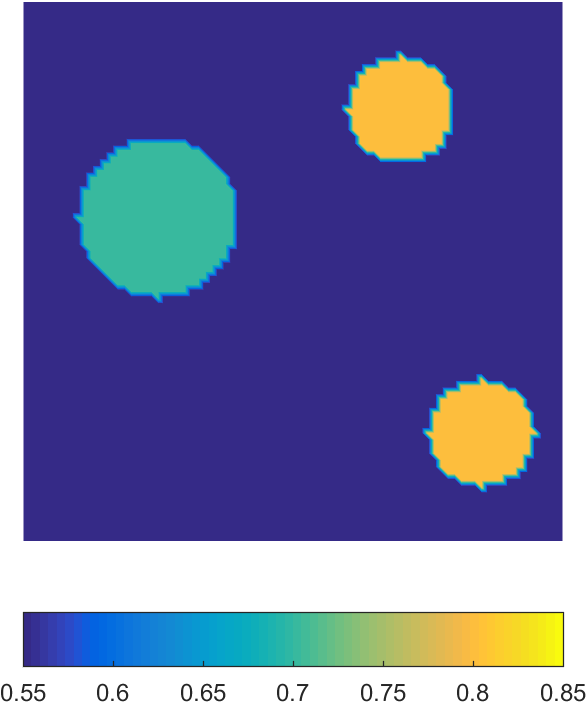} 
	\includegraphics[angle=0,width=0.24\textwidth]{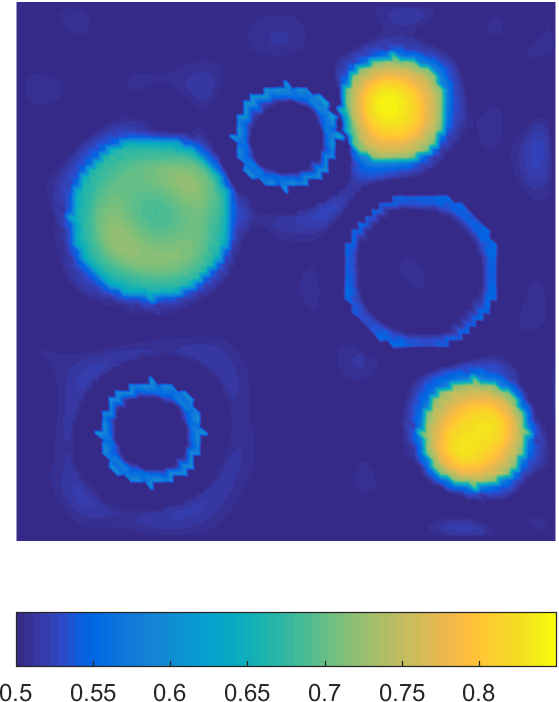}
	\includegraphics[angle=0,width=0.24\textwidth]{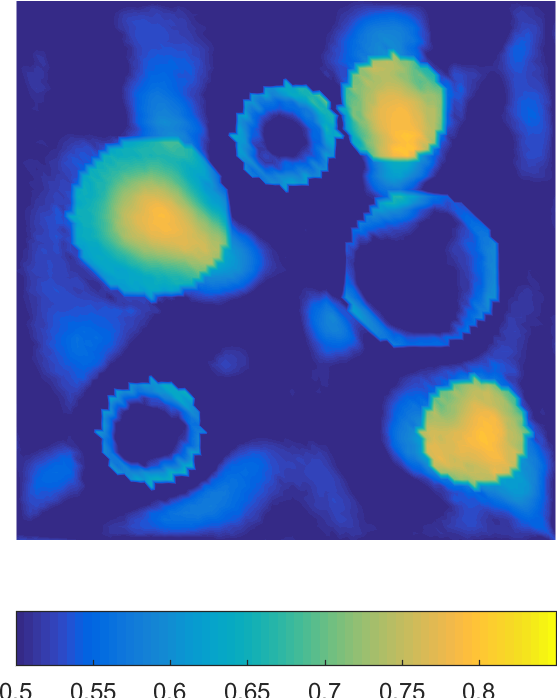}
	\includegraphics[angle=0,width=0.24\textwidth]{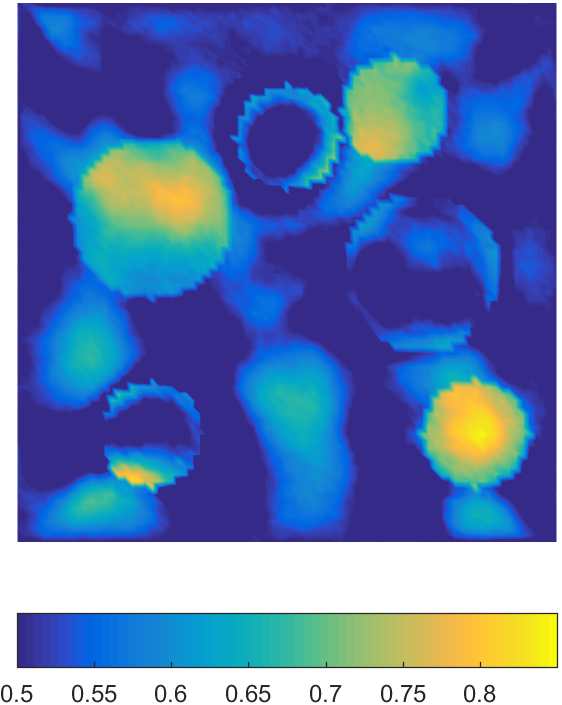} 
	\\
	\includegraphics[angle=0,width=0.24\textwidth]{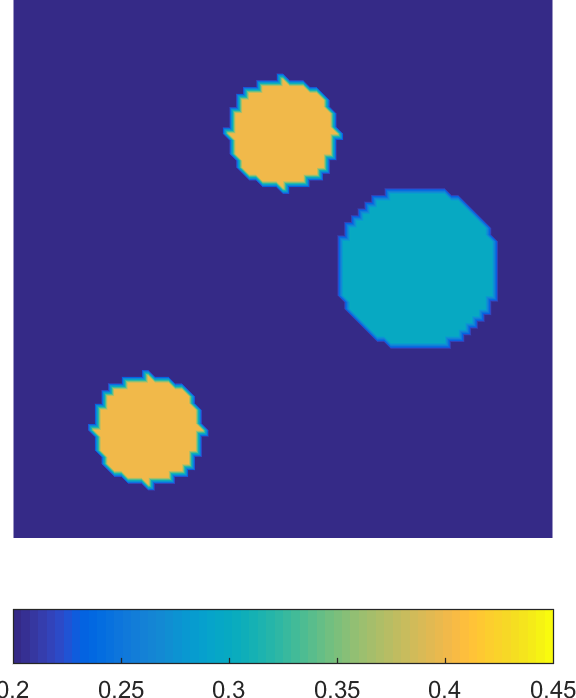}
	\includegraphics[angle=0,width=0.24\textwidth]{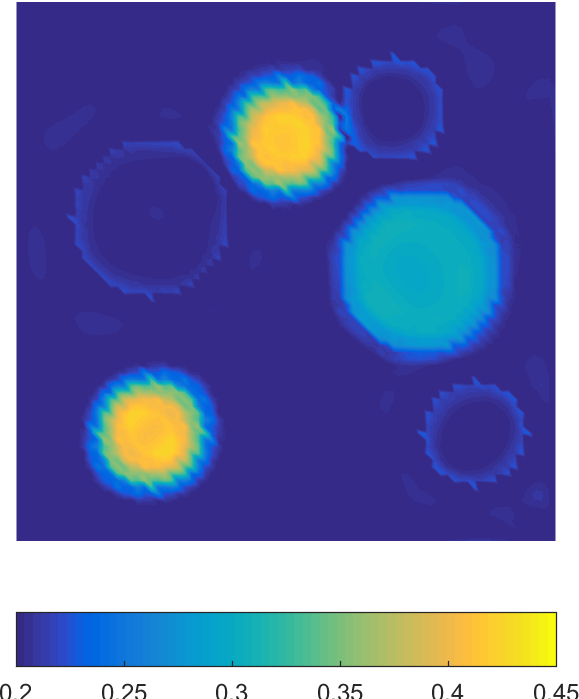}
	\includegraphics[angle=0,width=0.24\textwidth]{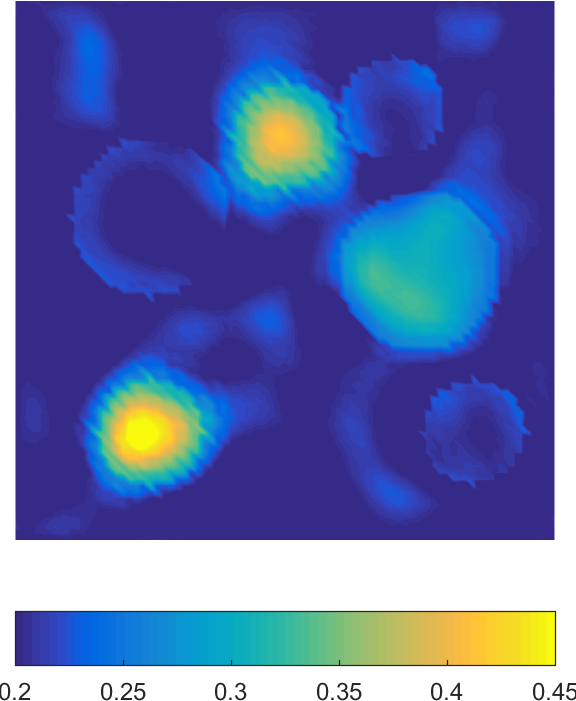}
	\includegraphics[angle=0,width=0.24\textwidth]{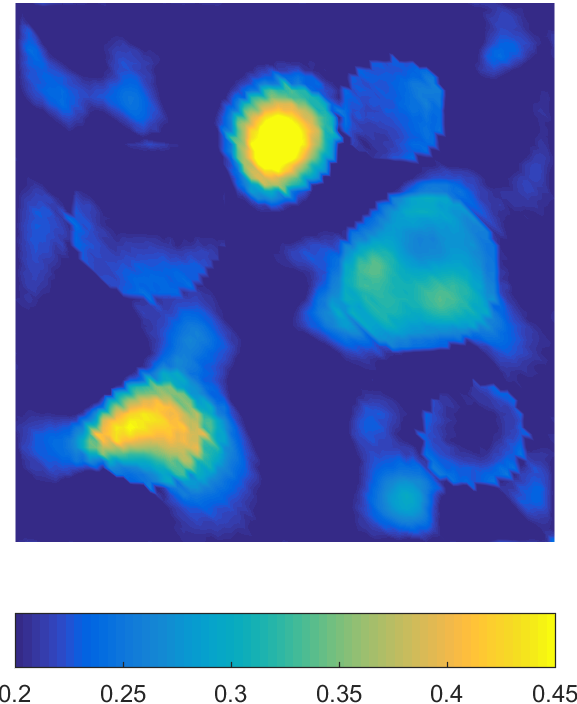}
	\caption{The true optical coefficient pair $\fo=(\Gamma, \sigma_a)$ (left), the mean of the reconstructed pair $(\wh\Gamma_0, \wh \sigma_0)$ (second column) and two realizations of the reconstructions formed from the reconstructed PCE coefficients (the two columns on the right).}
	\label{FIG:GS-v-C Recon}
\end{figure}
\paragraph{Experiment I. [Ultrasound Speed Uncertainty in PAT]} In the first numerical experiment, we attempt to reconstruct the optical coefficient pair $\fo=(\Gamma, \sigma_a)$ from ultrasound data sets generated from four different illumination sources. We set the true sound speed to be the constant $c_0(\bx)=1.0$ and generate random realizations of the ultrasound speed around this value by selecting appropriate PCE coefficients according to~\eqref{EQ:Smooth Uncert Coeff}. The random perturbations created are therefore smooth in space. We take $K_c=12$ PCE modes in the construction after numerical tests showed that increasing $K_c$ does not change the simulation results significantly anymore; see the top row of Figure~\ref{FIG:Ultrasound Speed} for some typical realizations of  the ultrasound speed in this setup. 

In Figure~\ref{FIG:GS-v-C Recon} we show the true coefficients, the average of the reconstructed coefficients (that is, $(\wh\Gamma_0,\ \wh\sigma_0)$) and two realizations of the reconstructed coefficients (that is, $(\Gamma, \sigma_a)$ that we formed from the reconstructed PCE coefficients using the approximation~\eqref{EQ:PCE fu}). We observe that the average of the reconstructions, $(\wh\Gamma_0,\ \wh\sigma_0)$, is very close to the true coefficients as it should be (see, for instance, previously published results in~\cite{DiReVa-IP15}), and the variance, as seen from the two realizations on the right two columns in Figure~\ref{FIG:GS-v-C Recon}, is fairly large. To quantitatively measure the impact of the uncertainty in ultrasound speed on the reconstruction of the optical coefficients, we look at the (relative) standard deviation of the reconstruction as a function of the (relative)  standard deviation of the uncertainty coefficients. More precisely, we define
\[
	\cE_\fo \equiv \dfrac{\left\|\sqrt{\sum_{k=1}^{K_\fo} |\wh\fo_k|^2}\right\|_{L^2(X)}}{\|\wh \fo_0\|_{L^2(X)}} \qquad \mbox{and}\qquad \cE_\fu \equiv \dfrac{\left\|\sqrt{\sum_{k=1}^{K_\fu} |\wh\fu_k|^2}\right\|_{L^2(X)}}{\|\wh \fu_0\|_{L^2(X)}}
\]
for the objective coefficients (to be reconstructed) and the uncertainty coefficients respectively. Note that we have integrated all quantities over the domain to get numbers instead of functions since we don't have better ways to visualize the dependence.

In Figure~\ref{FIG:GS v C Error}, Figure~\ref{FIG:gaS v C Error} and Figure~\ref{FIG:Gga v C Error} we show the uncertainty level in the reconstructed objective coefficients versus the uncertainty level in uncertainty coefficient (i.e. the ultrasound speed) in the case of $\fo=(\Gamma, \sigma_a)$, $\fo=(\gamma, \sigma_a)$ and $\fo=(\Gamma, \gamma)$, respectively. We observe that in all three cases, when the uncertainty level in the ultrasound speed, measured by $\cE_c$, is small, it has roughly linear impact on the reconstructions. When the uncertainty level becomes larger, its impact becomes super-linear, but still very controllable. We do not have sufficient computational power to get enough data points to reliably fit an accurate curve between $\cE_\fo$ and $\cE_\fu$. However, the general relation between $\cE_\fo$ and $\cE_\fu$ is obvious enough to be observed in the existing simulation data.
\begin{figure}[ht]
\centering
\includegraphics[width=0.3\textwidth,height=0.2\textheight]{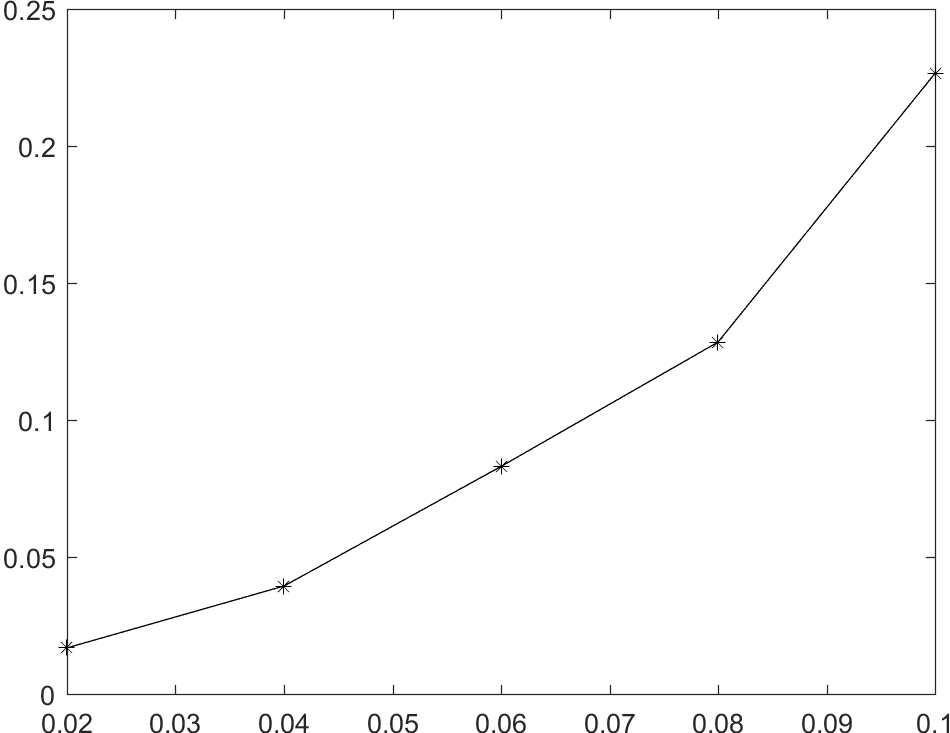}\hskip 1cm
\includegraphics[width=0.3\textwidth,height=0.2\textheight]{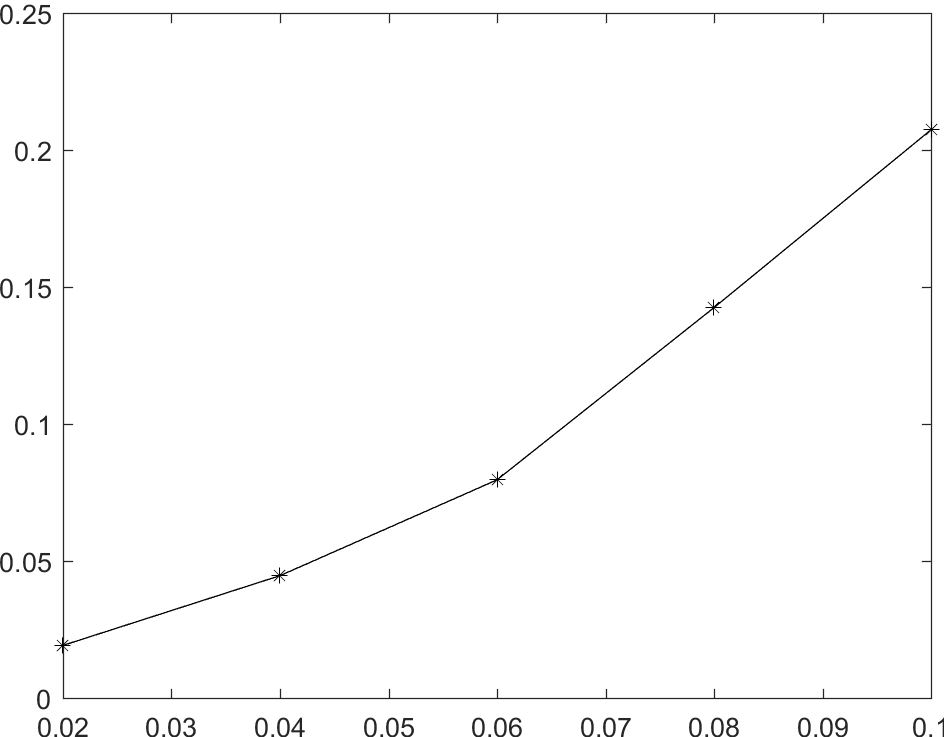}
\caption{Relative standard deviation of the objective coefficient pair $\fo=(\Gamma, \sigma_a)$, $\cE_\fo$, versus the relative standard deviation of the uncertainty coefficient $\fu=c$, $\cE_\fu$, for Experiment I.}
\label{FIG:GS v C Error}
\end{figure}
\begin{figure}[ht]
\centering
\includegraphics[width=0.3\textwidth,height=0.2\textheight]{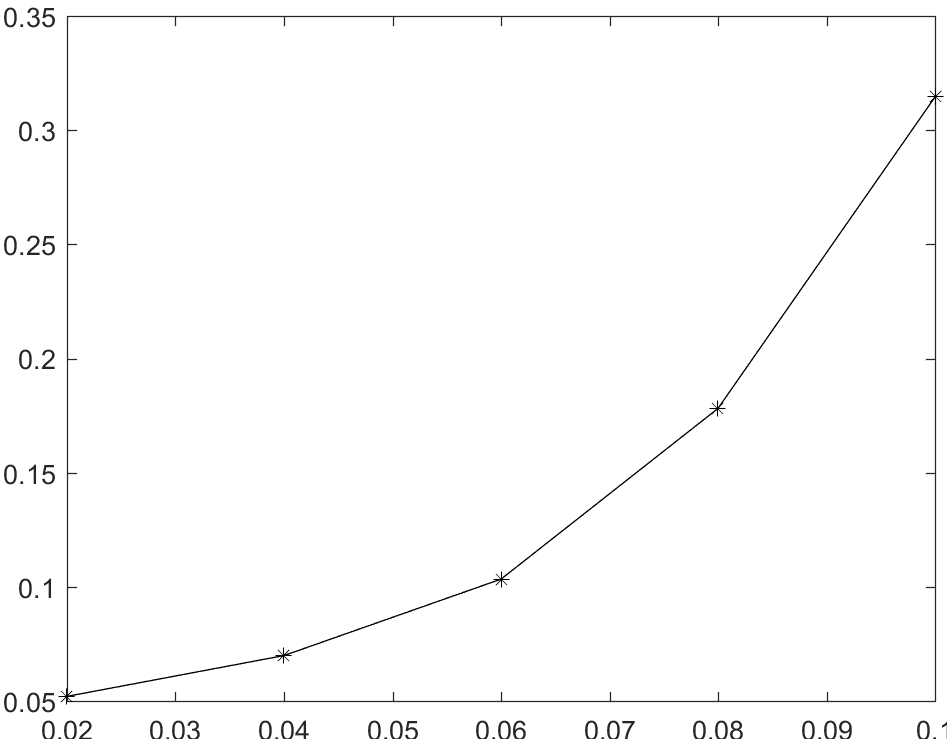}\hskip 1cm
\includegraphics[width=0.3\textwidth,height=0.2\textheight]{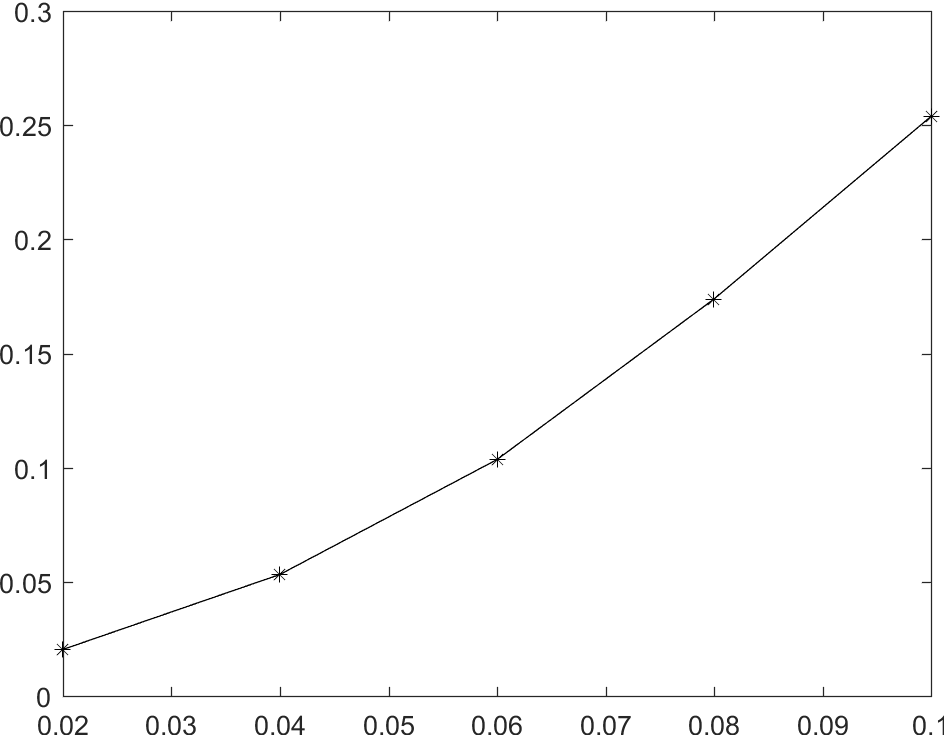}
\caption{Relative standard deviation of the objective coefficient pair $\fo=(\gamma, \sigma_a)$, $\cE_\fo$, versus the relative standard deviation of the uncertainty coefficient $\fu=c$, $\cE_\fu$, for Experiment I.}
\label{FIG:gaS v C Error}
\end{figure}
\begin{figure}[ht]
\centering
\includegraphics[width=0.3\textwidth,height=0.2\textheight]{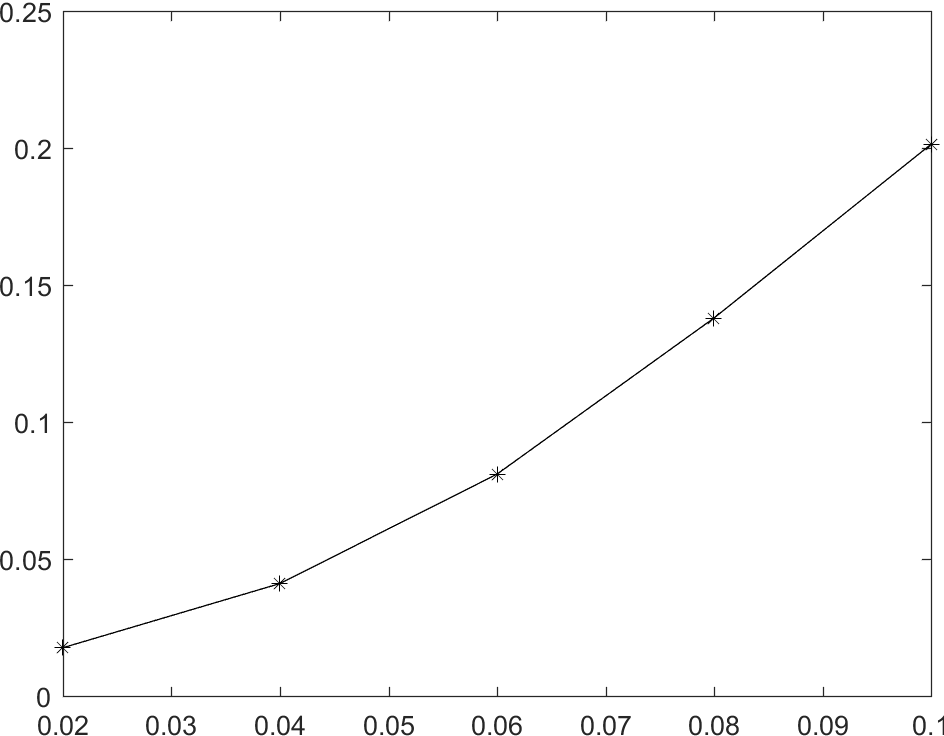}\hskip 1cm
\includegraphics[width=0.3\textwidth,height=0.2\textheight]{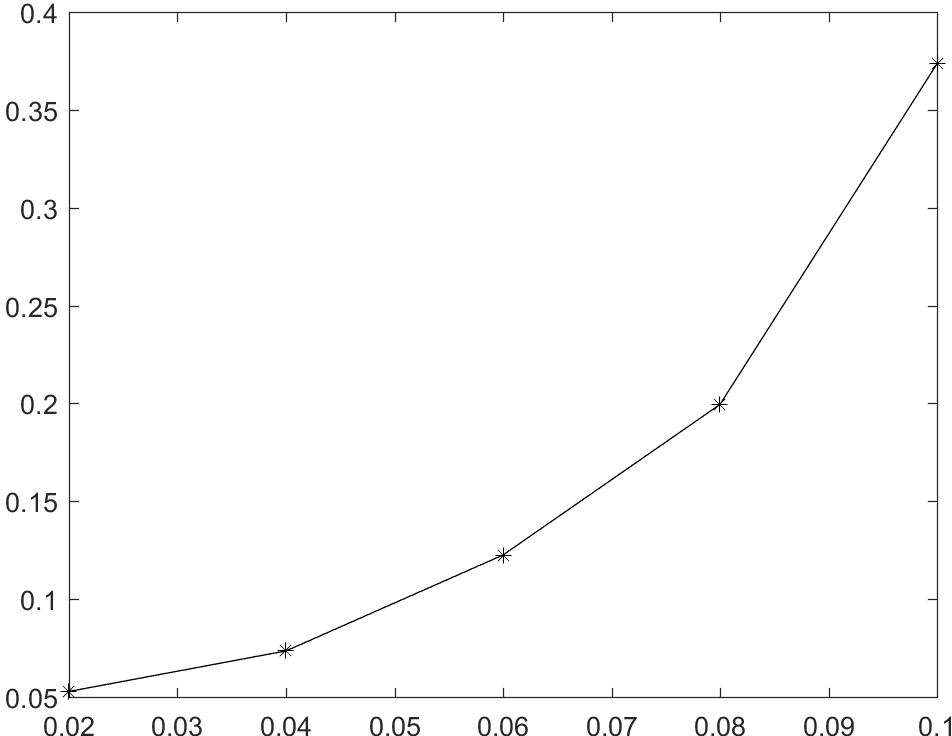}
\caption{Relative standard deviation of the objective coefficient pair $\fo=(\Gamma, \gamma)$, $\cE_\fo$, versus the relative standard deviation of the uncertainty coefficient $\fu=c$, $\cE_\fu$, for Experiment I.}
\label{FIG:Gga v C Error}
\end{figure}

We repeat the numerical simulations in Experiment I with piecewise smooth ultrasound speed constructed from~\eqref{EQ:Disc Uncert Coeff}. We use $K_c=12$ again in this simulation. In the bottom row of Figure~\ref{FIG:Ultrasound Speed}, we show four realizations of the ultrasound speed in this setup.  In Figure~\ref{FIG:GS v C-Rough Error}, Figure~\ref{FIG:gaS v C-Rough Error} and Figure~\ref{FIG:Gga v C-Rough Error}, we show the $\cE_\fo-\cE_\fu$ relations in the reconstructions of $\fo=(\Gamma, \sigma_a)$, $\fo=(\gamma, \sigma_a)$ and $\fo=(\Gamma, \gamma)$ respectively. We observe that even though the curves look like those in Figures~\ref{FIG:GS v C Error}, ~\ref{FIG:gaS v C Error} and ~\ref{FIG:Gga v C Error} for smooth random ultrasound speed, they are significantly different in the sense that piecewise smooth random ultrasound speed creates much larger impact on the reconstructions of the optical coefficients. We performed another set of simulations where the locations of the perturbations (i.e. the disks in~\eqref{EQ:Disc Uncert Coeff}) are randomly changed. The same increasing in the uncertainty of the reconstructions are observed.
\begin{figure}[ht]
\centering
\includegraphics[width=0.3\textwidth,height=0.2\textheight]{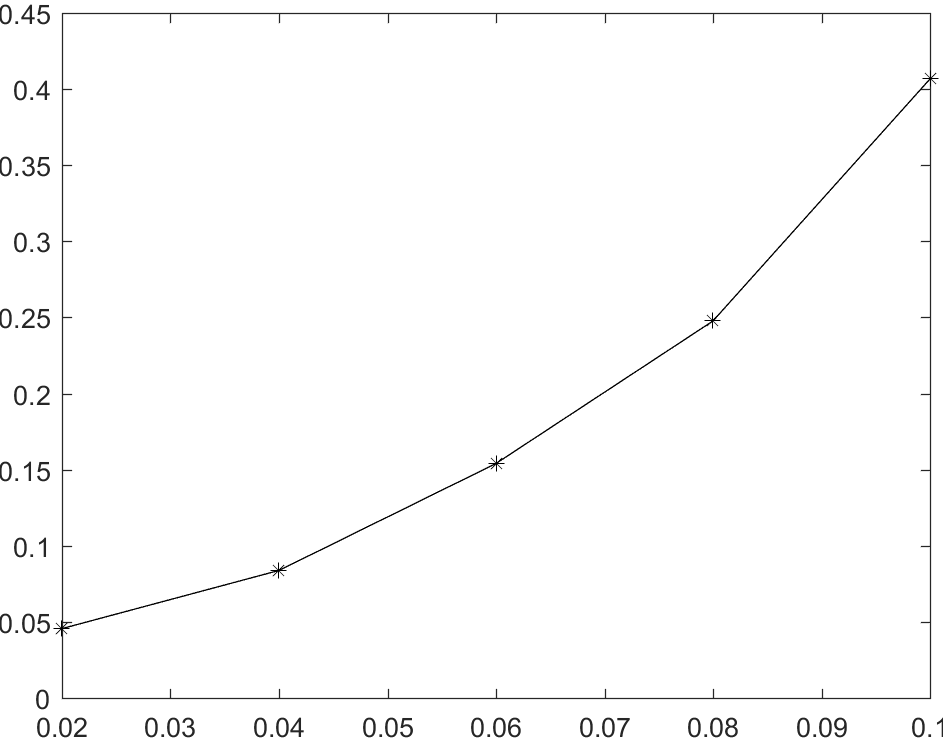} \hskip 1cm
\includegraphics[width=0.3\textwidth,height=0.2\textheight]{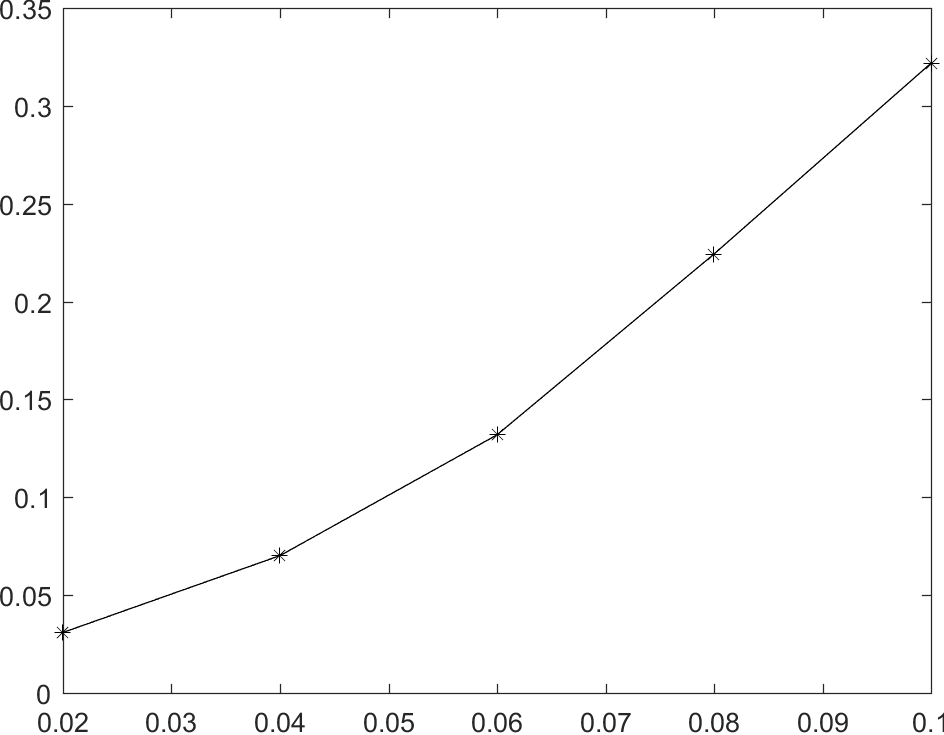} \hskip 1cm
\caption{Same as in Figure~\ref{FIG:GS v C Error} but for piecewise smooth ultrasound speed constructed from~\eqref{EQ:Disc Uncert Coeff}.}
\label{FIG:GS v C-Rough Error}
\end{figure}
\begin{figure}[ht]
\centering
\includegraphics[width=0.3\textwidth,height=0.2\textheight]{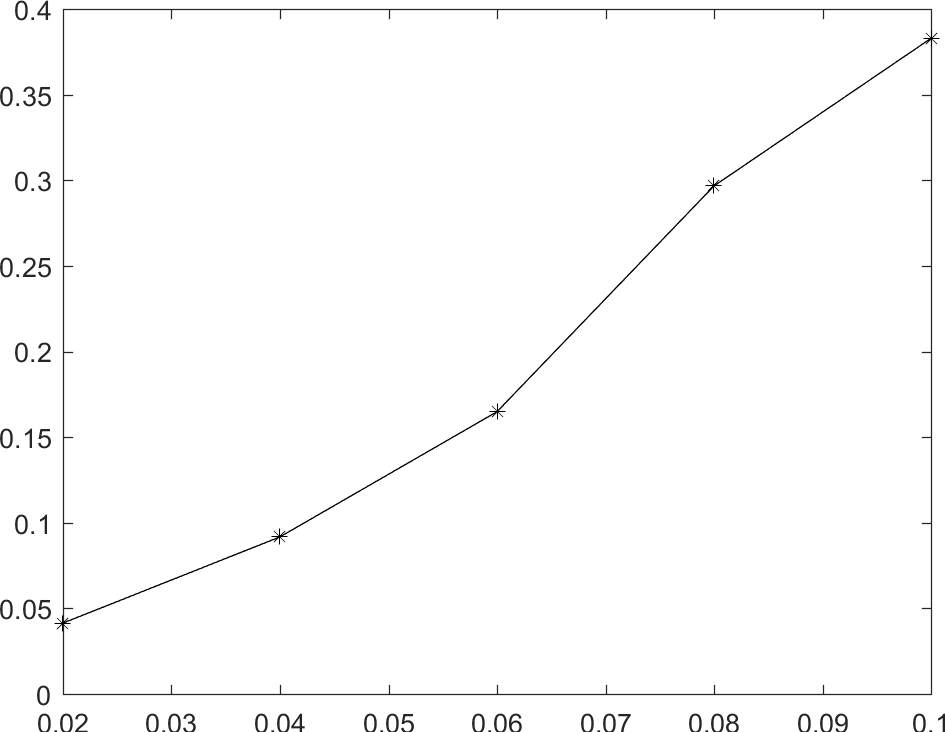} \hskip 1cm
\includegraphics[width=0.3\textwidth,height=0.2\textheight]{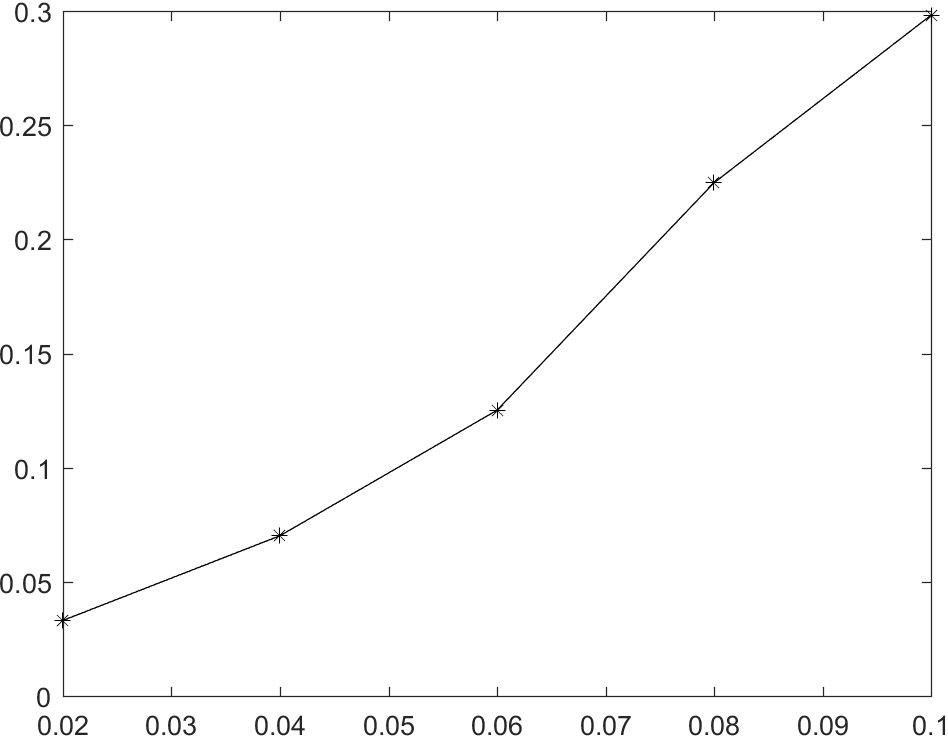}
\caption{Same as in Figure~\ref{FIG:gaS v C Error} but for piecewise smooth ultrasound speed constructed from~\eqref{EQ:Disc Uncert Coeff}.}
\label{FIG:gaS v C-Rough Error}
\end{figure}
\begin{figure}[ht]
\centering
\includegraphics[angle=0,width=0.3\textwidth,height=0.2\textheight]{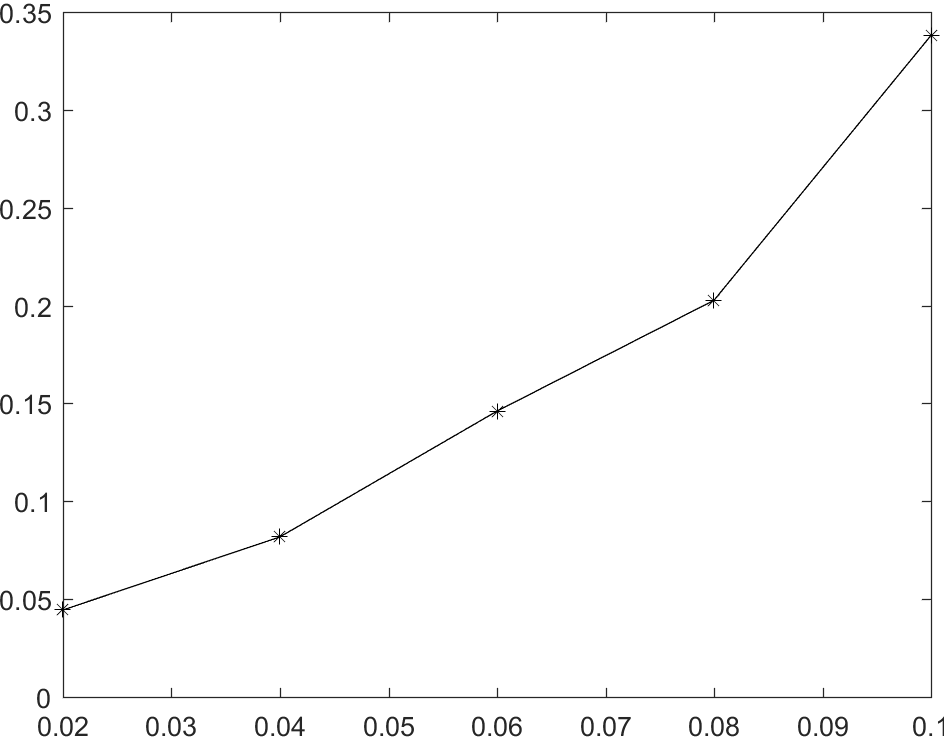}\hskip 1cm
\includegraphics[angle=0,width=0.3\textwidth,height=0.2\textheight]{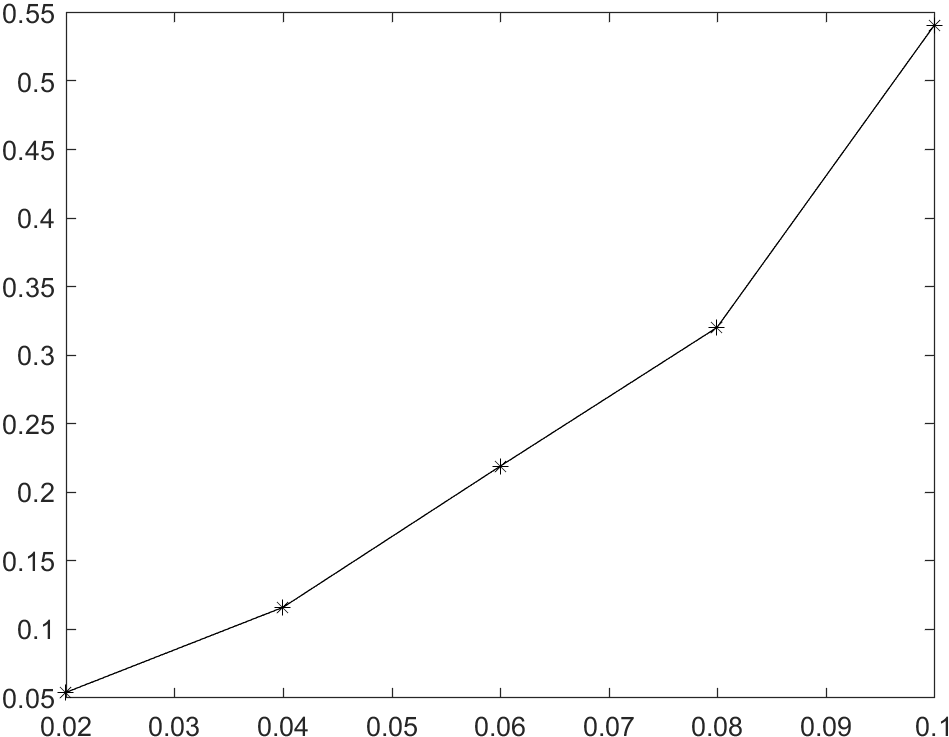}
\caption{Same as in Figure~\ref{FIG:Gga v C Error} but for piecewise smooth ultrasound speed constructed from~\eqref{EQ:Disc Uncert Coeff}.}
\label{FIG:Gga v C-Rough Error}
\end{figure}

\begin{figure}[htb!]
	\centering
	\includegraphics[angle=0,width=0.24\textwidth]{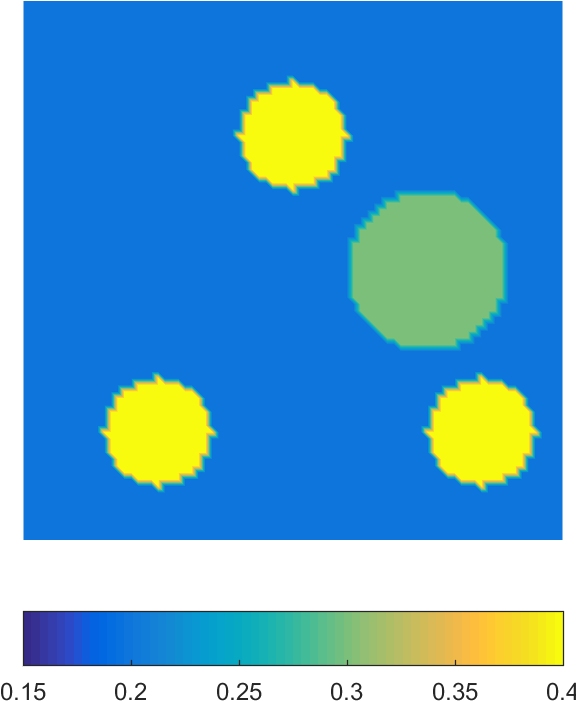} 
	\includegraphics[angle=0,width=0.24\textwidth]{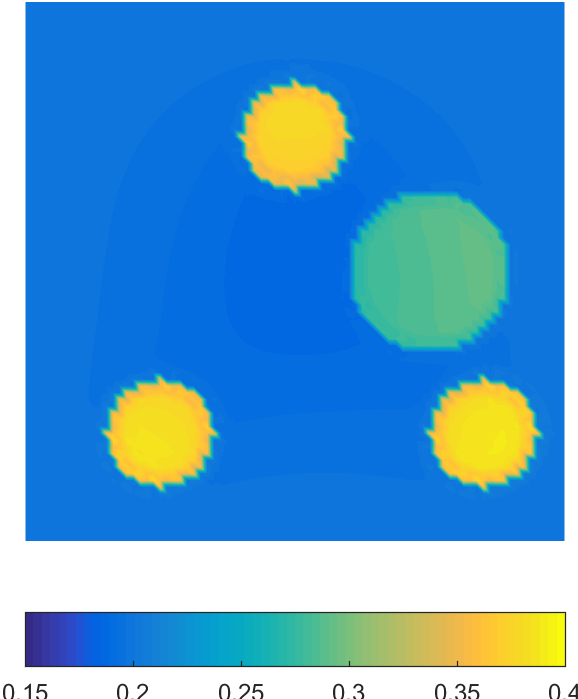}
	\includegraphics[angle=0,width=0.24\textwidth]{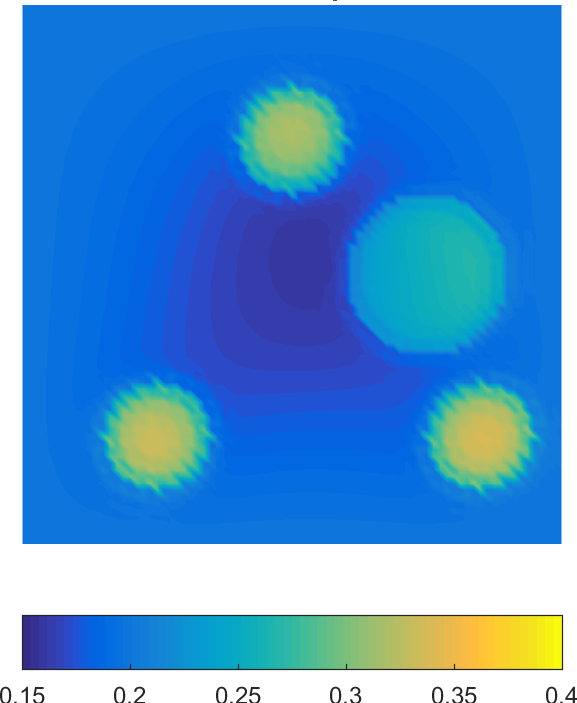}
	\includegraphics[angle=0,width=0.24\textwidth]{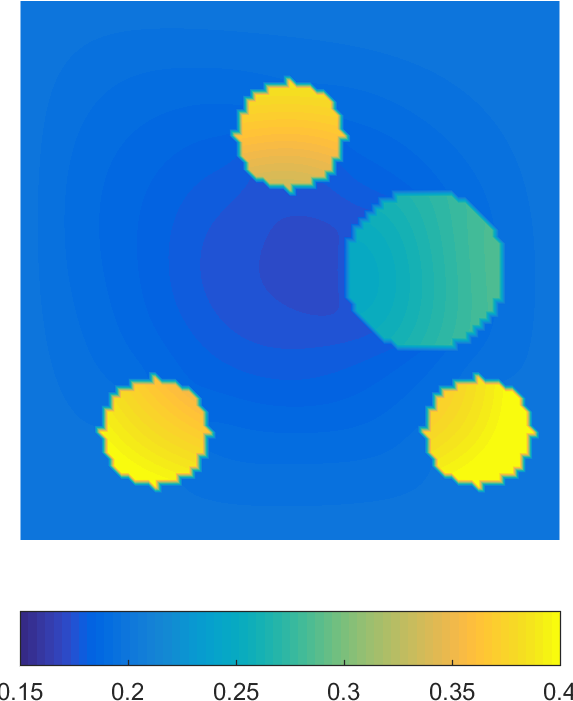}
	\caption{The true optical coefficient $\fo=\sigma_{a,xf}$ (left), the mean of the reconstructions $\wh \sigma_0$ (second column) and two realizations of the reconstructions formed from the reconstructed PCE coefficients (the two columns on the right) in fPAT.}
	\label{FIG:fPAT Sigmaxf-v-C Recon}
\end{figure}
\paragraph{Experiment II. [Ultrasound Speed Uncertainty in fPAT]} In this numerical experiment, we characterize uncertainty in the reconstruction of the fluorescence absorption coefficient $\fo=\sigma_{a,xf}$ in fluorescence PAT caused by uncertainty in the ultrasound speed. We collect ultrasound data generated from two different illumination sources. We again perform simulations with both smooth ultrasound speed from~\eqref{EQ:Smooth Uncert Coeff} and piecewise smooth ultrasound speed from~\eqref{EQ:Disc Uncert Coeff}. In both cases, we take $K_c=12$. In Figure~\ref{FIG:fPAT Sigmaxf-v-C Recon} we show the true absorption coefficient $\sigma_{a,xf}$, the mean of the reconstruction of it and two realizations of the reconstructions formed from the reconstructed PCE coefficients. We observe again that the averaged reconstruction is very accurate, comparable to the numerical simulations in~\cite{ReZh-SIAM13,ReZhZh-IP15}. The uncertainty in the reconstructions depends on the uncertainty in the ultrasound speed as in the PAT case in Experiment I: piecewise smooth random ultrasound speed could produce larger uncertainty in the reconstructions than smooth random ultrasound speed; see the top and bottom rows of Figure~\ref{FIG:fPAT Sigmaxf v C Error} for a comparison.
\begin{figure}[htb!]
\centering
\includegraphics[width=0.3\textwidth,height=0.2\textheight]{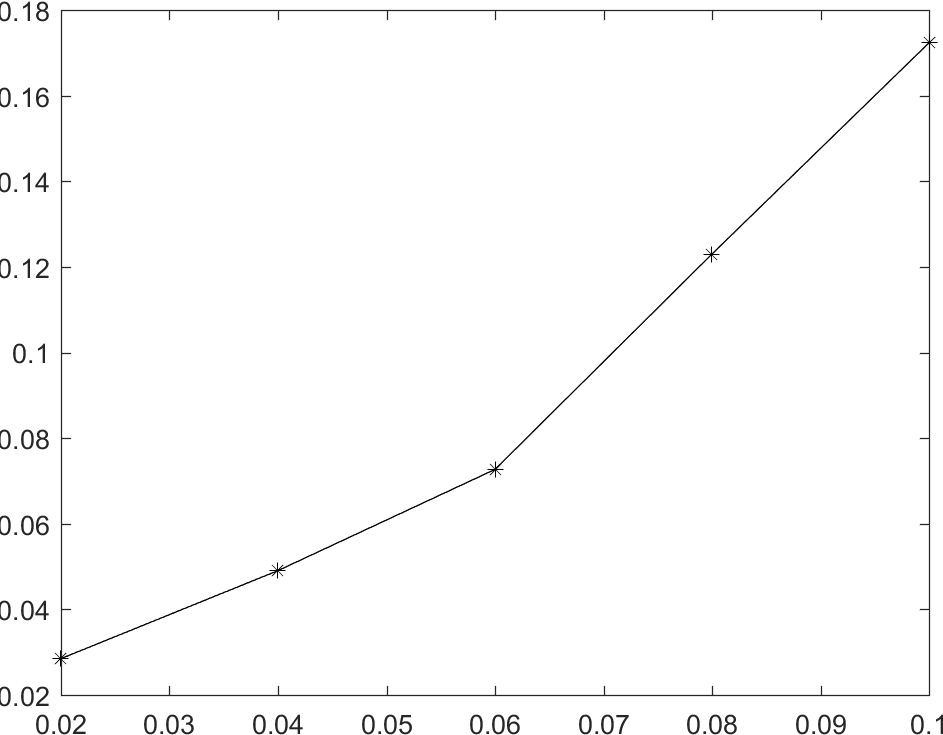}\hskip 1cm
\includegraphics[width=0.3\textwidth,height=0.2\textheight]{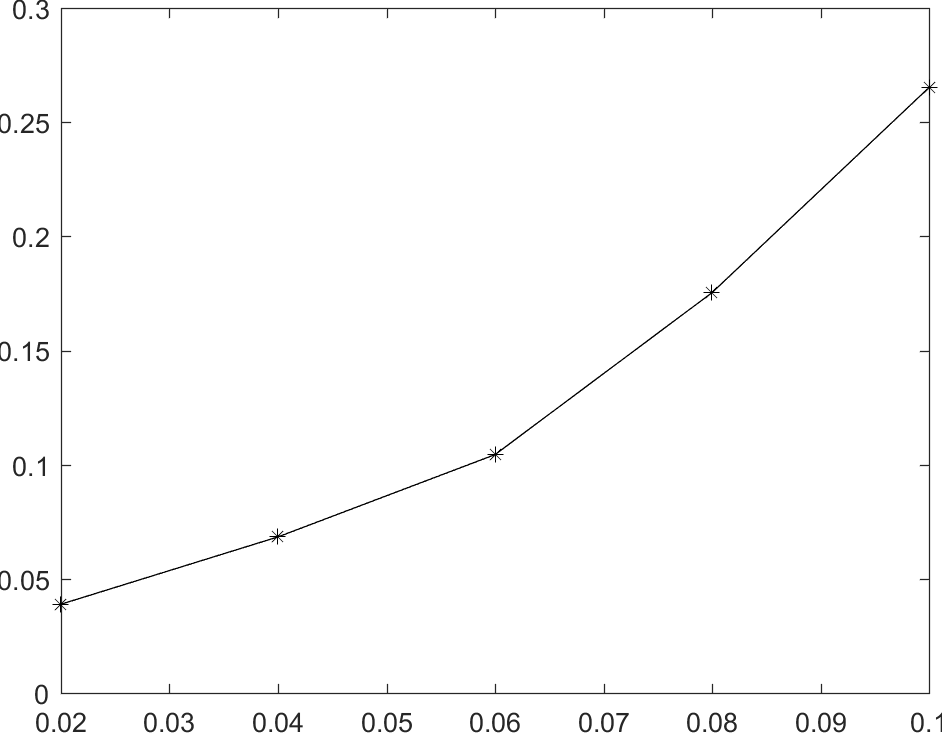}
\caption{Relative standard deviation of the objective coefficient $\fo=\sigma_{a,xf}$, $\cE_\fo$, versus the relative standard deviation of the uncertainty coefficient $\fu=c$, $\cE_\fu$, for Experiment II for smooth (left) and piecewise smooth (right) random ultrasound speed.}
\label{FIG:fPAT Sigmaxf v C Error}
\end{figure}

\subsection{Diffusion coefficient uncertainty}

We now characterize the uncertainty in optical reconstruction caused by uncertainty in the diffusion coefficient $\gamma$. In this case, the ultrasound speed is fixed in the data generation and inversion process. To avoid mixing the impact of errors in numerical wave propagation (and back-propagation) with impact of uncertainty of the diffusion coefficient, we start directly from internal data. That is, we only consider the uncertainty propagation from $\gamma$ to the internal datum $H$ and then $H$ to the objective coefficients to be reconstructed.
 
\paragraph{Experiment III. [Diffusion Coefficient Uncertainty in PAT]} We consider the reconstruction of the coefficient pair $\fo=(\Gamma, \sigma_a)$ using internal data generated from four different illuminations. We again perform simulations for both smooth random diffusion coefficients  and piecewise smooth random diffusion coefficients, with $K_\gamma=8$ and $K_\gamma=12$ respectively. The $\cE_\fo$ and $\cE_\fu$ relations are shown in Figure~\ref{FIG:GS v ga Error}. 
\begin{figure}[htb!]
\centering
\includegraphics[width=0.3\textwidth,height=0.2\textheight]{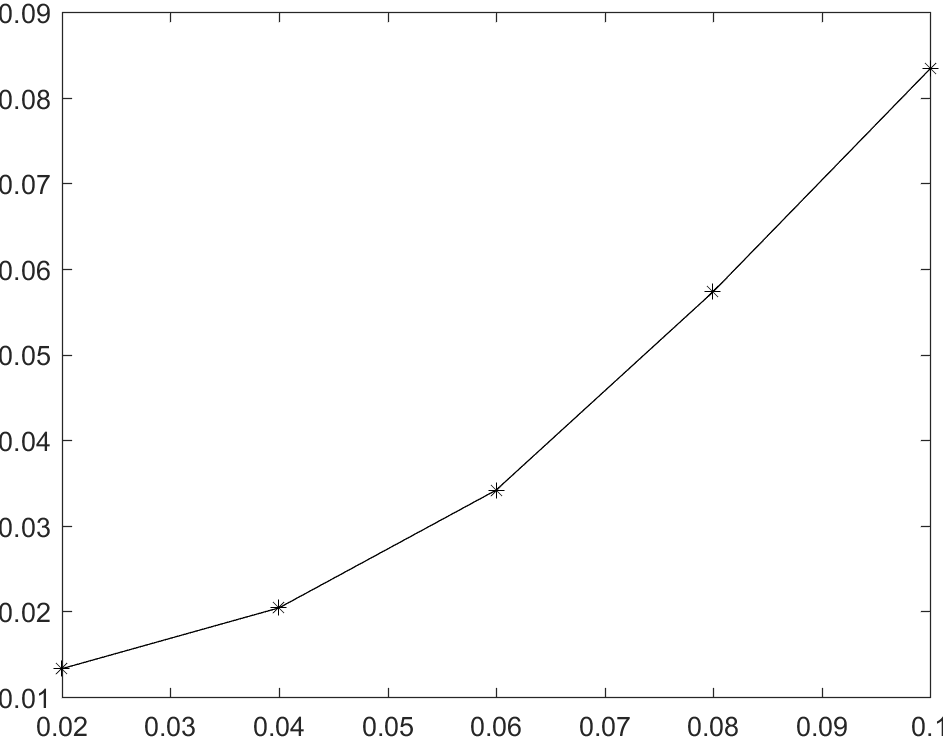}\hskip 1cm
\includegraphics[width=0.3\textwidth,height=0.2\textheight]{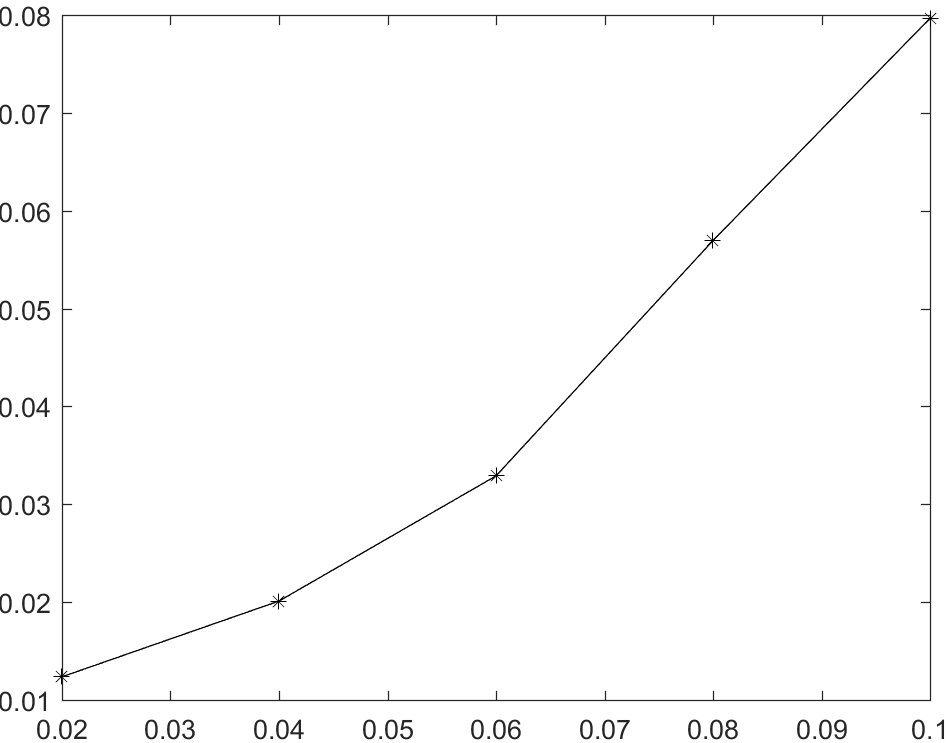}
\caption{Relative standard deviation of the objective coefficient pair $\fo=(\Gamma, \sigma_a)$, $\cE_\fo$, versus the relative standard deviation of the uncertainty coefficient $\fu=\gamma$, $\cE_\fu$, for Experiment III for smooth (top row) and piecewise smooth random $\gamma$.}
\label{FIG:GS v ga Error}
\end{figure}
In all the simulations, the true diffusion coefficient is taken as the constant $\wh\gamma_0=0.02$. We performed simulations at with other true diffusion coefficients. The results are very similar to those presented in Figure~\ref{FIG:GS v ga Error}.

The results demonstrate here again that uncertainty in piecewise smooth diffusion coefficients have larger impact on that in smooth diffusion coefficient. However, comparing Figure~\ref{FIG:GS v ga Error} with Figure~\ref{FIG:GS v C Error} and Figure~\ref{FIG:GS v C-Rough Error} shows that uncertainty in the diffusion coefficient has much smaller impact on the reconstruction of $(\Gamma, \sigma_a)$ than that in the ultrasound speed.
 
\subsection{Model uncertainty in fPAT}

In the last numerical experiment, we quantify the error in the reconstruction of the fluorescence absorption coefficient $\sigma_{a,xf}$ caused by the partial linearization, that is, dropping the coefficient $\sigma_{a,xf}$ in the first equation, of the diffusion model~\eqref{EQ:Diff fPAT}.

\paragraph{Experiment IV. [Model uncertainty in fPAT]} In our numerical simulations, we fixed every coefficient besides the fluorescence coefficient $\sigma_{a,xf}$. In this case, one well-chosen internal datum ~\eqref{EQ:Data QfPAT} allows unique and stable reconstruction of $\sigma_{a,xf}$~\cite{ReZh-SIAM13}. We generate the synthetic data from four different illuminations located on the four sides of the domain respectively, using the full diffusion model~\eqref{EQ:Diff fPAT}. We perform numerical reconstructions of $\sigma_{a,xf}$ using both the full diffusion model and the partially linearized diffusion model, i.e. the diffusion system~\eqref{EQ:Diff fPAT} without $\sigma_{a,xf}$ in the first equation. Let us denote by $\sigma_{a,xf}^{r}$ and $\sigma_{a,xf}^{r\ell}$ the reconstructions from the full diffusion model and the partially linearized model respectively, we compute the relative error caused by linearization as:
\[
	\cE_\ell=\frac{\|\sigma_{a,xf}^{r\ell}-\sigma_{a,xf}^r\|_{L^2( X)}}{\|\sigma_{a,xf}^r\|_{L^2( X)}}.
\]

We show in Figure~\ref{FIG:GS Linearization Recon} a true $\sigma_{a,xf}$, its reconstruction using the full diffusion model~\eqref{EQ:Diff fPAT} with noise-free data and noisy data, and its reconstruction with the partially linearized diffusion model. The reconstruction with the full diffusion model is very accurate, even when data is polluted with a little random noise, but the reconstruction with the partially linearized model is much less accurate, despite of the fact that the singularity in the coefficient is well reconstructed (since it is directly encoded in the internal data).
\begin{figure}[htb!]
\centering
\includegraphics[width=0.24\textwidth]{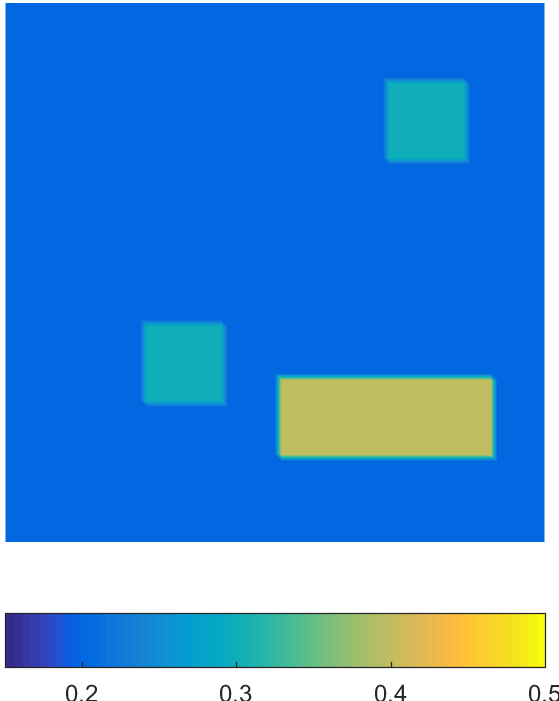}
\includegraphics[width=0.24\textwidth]{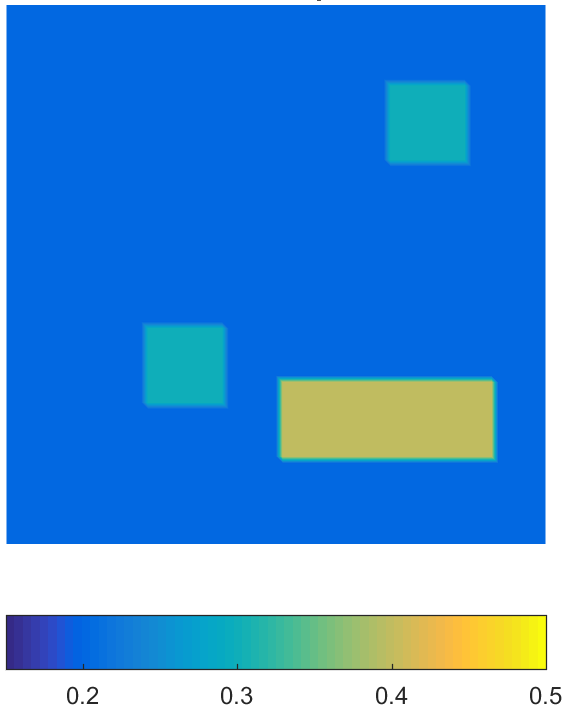}
\includegraphics[width=0.24\textwidth]{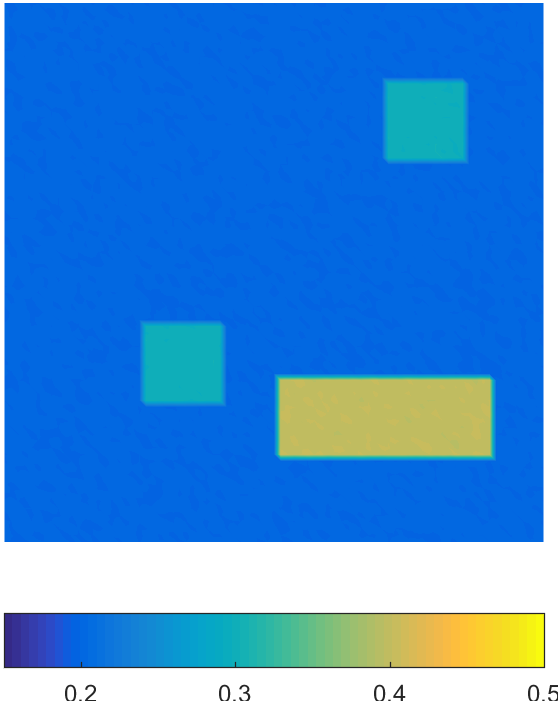}
\includegraphics[width=0.24\textwidth]{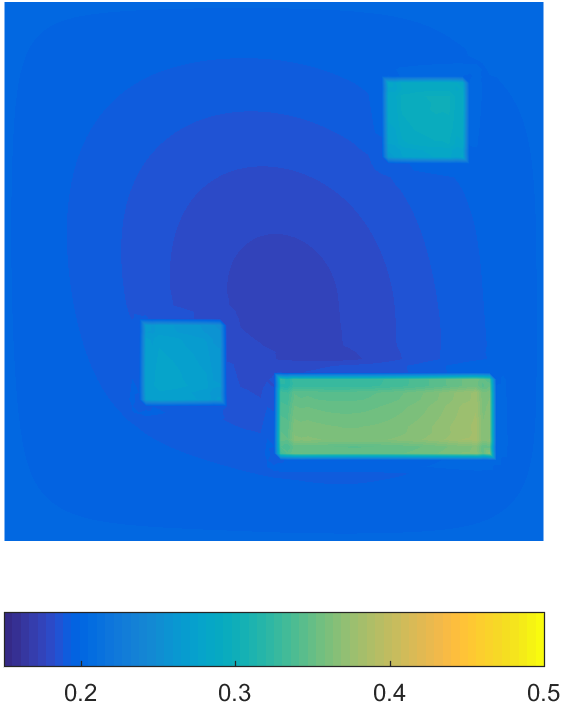}
\caption{True $\sigma_{a,xf}$ (first), its reconstruction with the full diffusion model~\eqref{EQ:Diff fPAT} using noisy free data (second) and data contain $2$\% uniformly distributed multiplicative noise (third), and its reconstruction with partially linearized diffusion model with noise free data (fourth).}
\label{FIG:GS Linearization Recon}
\end{figure}
We performed reconstructions for four different true $\sigma_{a,xf}$. The relative errors are respectively $\cE_\ell=0.08, \cE_\ell=0.09, \cE_\ell=0.08$ and $\cE_\ell=0.07$. These results show that the impact of the partial linearization on the reconstruction of the coefficient $\sigma_{a,xf}$ is relatively large. Therefore, even the partial linearization simplifies the solution of the diffusion model~\eqref{EQ:Diff fPAT}, for the sake of accuracy in reconstructions, it is probably a simplification that should not be performed in fPAT.

\section{Concluding Remarks}
\label{SEC:Concl}

In this work, we performed some analytical and numerical studies on the impact of uncertain model coefficients on the quality of the reconstructed images in photoacoustic tomography and fluorescence photoacoustic tomography. Particularly, we derived bounds on errors in the reconstruction of optical properties caused by errors in ultrasound speed used in the reconstructions, as well as bounds on error in the reconstruction of the fluorescence absorption coefficient in fPAT due to inaccuracy in the light propagation model caused by partial linearization. We presented a numerical procedure for the quantitative evaluation of such errors and performed computational simulations following the numerical procedure.

Our numerical simulations in PAT reconstructions show two phenomena that are prominent. The first is that in general, uncertainties in rougher ultrasound speed can produce larger uncertainty in reconstructed optical coefficients than what a smoother ultrasound speed can. This agrees with the general belief among researchers that reconstruction of the internal datum $H$ is ``stabler'' when the underline ultrasound speed is smooth. The second phenomenon is that in general, variations in ultrasound speed $c(\bx)$ can have much larger impact on the reconstruction of optical coefficients than variations in the diffusion coefficient $\gamma$ in the system. For the fPAT reconstructions, we observe numerically that the partial linearization by setting the fluorescence absorption coefficient $\sigma_{a,xf}=0$ in the x-component of the diffusion model~\eqref{EQ:Diff fPAT} can produce large error in the reconstruction of $\sigma_{a,xf}$. 

It is obvious that the uncertainty in the reconstruction depends on both the uncertainty in the model, which induces uncertainty in the data used for the reconstruction, and the method of the reconstructions as we explained in the Introduction (right below~\eqref{EQ:Model Gen Lin-2}). Due to the fact that we used the $l^2$ least-square optimization method for the reconstruction of the objective coefficients, which means the reconstructions are likely made smoother than they should be, the uncertainty numbers that we have seen might be actually slightly smaller than they should be. However, this effect should not distort significantly the overall trends we have observed numerically.

Characterization of errors in reconstructions caused by uncertainties in system parameters is an important task for many inverse problems in hybrid imaging modalities, or more generally any model-based imaging methods. The general methodology we developed in this work can be generalized to these inverse problems in a straightforward manner. The results we have can be generalized to deal with the situation when additional measurement noise are presented. In that case, the general model~\eqref{EQ:Model Gen} becomes $y^e=f(\fo,\fu)+e$, $e$ being the measurement noise, and the interplay between impact of $\fu$ and that of $e$ need to be analyzed carefully. We plan to investigate in this direction in a future work.


\section*{Acknowledgments}

This work is partially supported by the National Science Foundation through grant DMS-1620473. S.V. would also like to acknowledge partial support from the Statistical and Applied Mathematical Sciences Institute (SAMSI).

{\small

}

\end{document}